\documentclass[reqno]{amsart}

\usepackage[all]{xy}
\usepackage{graphicx}

\usepackage{amssymb}
\usepackage{amsmath}
\usepackage{mathrsfs}
\usepackage{epsfig}
\usepackage{amscd}
\usepackage{color}

\usepackage{marginnote}
\usepackage{mathtools}
\usepackage{tikz-cd}

\usepackage{bookmark}
\usepackage[titletoc]{appendix} 
\hypersetup{hypertexnames=false} 

\def\E{\ifmmode{\mathbb E}\else{$\mathbb E$}\fi} 
\def\N{\ifmmode{\mathbb N}\else{$\mathbb N$}\fi} 
\def\R{\ifmmode{\mathbb R}\else{$\mathbb R$}\fi} 
\def\Q{\ifmmode{\mathbb Q}\else{$\mathbb Q$}\fi} 
\def\C{\ifmmode{\mathbb C}\else{$\mathbb C$}\fi} 
\def\H{\ifmmode{\mathbb H}\else{$\mathbb H$}\fi} 
\def\Z{\ifmmode{\mathbb Z}\else{$\mathbb Z$}\fi} 
\def\P{\ifmmode{\mathbb P}\else{$\mathbb P$}\fi} 
\def\T{\ifmmode{\mathbb T}\else{$\mathbb T$}\fi} 
\def\SS{\ifmmode{\mathbb S}\else{$\mathbb S$}\fi} 
\def\DD{\ifmmode{\mathbb D}\else{$\mathbb D$}\fi} 

\newcommand{\del}{\partial}
\newcommand{\Cont}{{\operatorname{Cont}}}

\newcommand{\ben}{\begin{enumerate}}
\newcommand{\een}{\end{enumerate}}
\newcommand{\be}{\begin{equation}}
\newcommand{\ee}{\end{equation}}
\newcommand{\bea}{\begin{eqnarray}}
\newcommand{\eea}{\end{eqnarray}}
\newcommand{\beastar}{\begin{eqnarray*}}
\newcommand{\eeastar}{\end{eqnarray*}}
\newcommand{\bc}{\begin{center}}
\newcommand{\ec}{\end{center}}

\theoremstyle{theorem}
\newtheorem{thm}{Theorem}[section]
\newtheorem{cor}[thm]{Corollary}
\newtheorem{lem}[thm]{Lemma}
\newtheorem{prop}[thm]{Proposition}

\theoremstyle{definition}
\newtheorem{defn}[thm]{Definition}
\newtheorem{rem}[thm]{Remark}
\newtheorem{ques}[thm]{Question}

\newtheorem{exm}[thm]{Example}
\newtheorem{hypo}[thm]{Hypothesis}

\newtheorem{notation}[thm]{\rm\bfseries{Notation}}

\newtheorem*{thm*}{Theorem}

\numberwithin{equation}{section}

\def\R{{\mathbb R}}

\def\Crit{{\hbox{Crit}}}

\def\E{{\mathbb E}}
\def\Z{{\mathbb Z}}
\def\C{{\mathbb C}}
\def\R{{\mathbb R}}
\def\P{{\mathbb P}}

\def\N{{\mathbb N}}

\def\11{{\mathbb I}}

\def\delbar{{\overline \partial}}

\def\C{\mathbb{C}}
\def\Z{\mathbb{Z}}

\def\T{\mathbb{T}}

\def\Q{\mathbb{Q}}
\def\K{\mathbb{K}}

\def\E{\ifmmode{\mathbb E}\else{$\mathbb E$}\fi} 
\def\N{\ifmmode{\mathbb N}\else{$\mathbb N$}\fi} 
\def\R{\ifmmode{\mathbb R}\else{$\mathbb R$}\fi} 
\def\Q{\ifmmode{\mathbb Q}\else{$\mathbb Q$}\fi} 
\def\C{\ifmmode{\mathbb C}\else{$\mathbb C$}\fi} 
\def\H{\ifmmode{\mathbb H}\else{$\mathbb H$}\fi} 
\def\Z{\ifmmode{\mathbb Z}\else{$\mathbb Z$}\fi} 
\def\P{\ifmmode{\mathbb P}\else{$\mathbb P$}\fi} 
\def\SS{\ifmmode{\mathbb S}\else{$\mathbb S$}\fi} 
\def\DD{\ifmmode{\mathbb D}\else{$\mathbb D$}\fi} 

\def\R{{\mathbb R}}

\def\Crit{{\hbox{Crit}}}
\def\E{{\mathbb E}}
\def\Z{{\mathbb Z}}
\def\C{{\mathbb C}}
\def\R{{\mathbb R}}

\def\N{{\mathbb N}}

\def\delbar{{\overline \partial}}

\def\CA{{\mathcal A}}

\def\CC{{\mathcal C}}
\def\CD{{\mathcal D}}
\def\CE{{\mathcal E}}
\def\CF{{\mathcal F}}

\def\CH{{\mathcal H}}

\def\CL{{\mathcal L}}
\def\CM{{\mathcal M}}

\def\CW{{\mathcal W}}

\def\CZ{{\mathcal Z}}

\def\LCI{\text{\rm LCI}}

\def\grad#1{\,\nabla\!_{{#1}}\,}

\def\darr#1{\raise1.5ex\hbox{$\leftrightarrow$}
\mkern-16.5mu #1}

\def\roughly#1{\raise.3ex\hbox{$#1$\kern-.75em
\lower1ex\hbox{$\sim$}}}

\def\opname#1{\mathop{\kern0pt{\rm #1}}\nolimits}

\def\Im{\opname{Im}}

\def\dim{\opname{dim}}

\def\grad{\opname{grad}}

\def\supp{\operatorname{supp}}

\def\Reeb{\operatorname{Reeb}}
\def\Aut{\operatorname{Aut}}

\def\span{\operatorname{span}}

\def\Cont{\operatorname{Cont}}
\def\Crit{\operatorname{Crit}}
\def\Spec{\operatorname{Spec}}
\def\Sing{\operatorname{Sing}}
\def\GFQI{\frak{G}}
\def\Index{\operatorname{Index}}

\def\Image{\operatorname{Image}}
\def\ev{\operatorname{ev}}

\def\Lag{\operatorname{Lag}}

\begin{document}

\quad \vskip1.375truein

\def\mq{\mathfrak{q}}
\def\mp{\mathfrak{p}}
\def\mH{\mathfrak{H}}
\def\mh{\mathfrak{h}}
\def\ma{\mathfrak{a}}
\def\ms{\mathfrak{s}}
\def\mm{\mathfrak{m}}
\def\mn{\mathfrak{n}}
\def\mz{\mathfrak{z}}
\def\mw{\mathfrak{w}}
\def\Hoch{{\tt Hoch}}
\def\mt{\mathfrak{t}}
\def\ml{\mathfrak{l}}
\def\mT{\mathfrak{T}}
\def\mL{\mathfrak{L}}
\def\mg{\mathfrak{g}}
\def\md{\mathfrak{d}}
\def\mr{\mathfrak{r}}
\def\Cont{\operatorname{Cont}}
\def\Crit{\operatorname{Crit}}
\def\Spec{\operatorname{Spec}}
\def\Sing{\operatorname{Sing}}
\def\GFQI{\text{\rm GFQI}}
\def\Index{\operatorname{Index}}
\def\Cross{\operatorname{Cross}}
\def\Ham{\operatorname{Ham}}
\def\Fix{\operatorname{Fix}}
\def\Graph{\operatorname{Graph}}
\def\LCIH{\operatorname{LCIH}}
\def\LCIC{\operatorname{LCIC}}
\def\Ob{\mathfrak{Ob}}
\def\Mor{\mathfrak{Mor}}
\def\co{\colon\thinspace}
\def\PO{\mathfrak{PO}}

\newcommand{\id}{{\bf 1}}

\title[Rational contact instantons]
{Rational contact instantons and Legendrian Fukaya category}
\author{Yong-Geun Oh}
\address{Center for Geometry and Physics, Institute for Basic Science (IBS),
77 Cheongam-ro, Nam-gu, Pohang-si, Gyeongsangbuk-do, Korea 790-784
\& POSTECH, Gyeongsangbuk-do, Korea}
\email{yongoh1@postech.ac.kr}

\begin{abstract} This is the first of a series of papers in preparation
on the Fukaya-type $A_\infty$
category generated by \emph{tame} Legendrian submanifolds, called the 
\emph{Legendrian contact instanton Fukaya category} 
(abbreviated as the Legendrian CI Fukaya category) and its applications to contact dynamics and topology. 
In the present paper, we give the construction of an $A_\infty$ category whose 
objects are Legendrian links and whose structure maps are
defined by the moduli spaces of \emph{finite energy contact instantons} on
tame contact manifolds in the sense of \cite{oh:entanglement1}. 
In a sequel \cite{kim-oh:long-exact}, jointed by Jongmyeong Kim, 
we will explain the relationships with various previous results in the literature
concerning Rabinowitz Fukaya categories \cite{CF,CFO}, \cite{ganatra-gao-venkat}, \cite{BJK}
on the Liouville manifolds with ideal boundary of  contact manifolds, and the Floer theory of 
Lagrangian cobordism \cite{CDGG}, \cite{EES:orientation}. 
\end{abstract}

\date{Nov 8, 2024}


\maketitle

\tableofcontents

\section{Introduction}

Let $(M,\xi)$ be an oriented contact manifold equipped with compatible coorientation, i.e.,
the volume form $\mu_\lambda: = \lambda \wedge (d\lambda)^n$ is compatible with the
given orientation of $M$. We
consider the (time-dependent) contact triad
$$
(M,\lambda, J), \quad J = \{J_t\}_{t \in [0,1]}.
$$
In particular we 
Let $H = H(t,x)$ be a time-dependent Hamiltonian. The following elliptic boundary value problem
for maps $u: \R \times [0,1] \to M$
\be\label{eq:perturbed-contacton-intro}
\begin{cases}
(du - X_H \otimes dt)^{\pi(0,1)} = 0, \quad d(e^{g_H(u)}(u^*\lambda + H\, dt)\circ j) = 0\\
u(\tau,0) \in R, \quad u(\tau,1) \in R
\end{cases}
\ee
is introduced in \cite{oh:contacton-Legendrian-bdy,oh:entanglement1}. Here
the function $g_H: \R\times [0,1] \to \R$ is defined by
\be\label{eq:gHu-intro}
g_H(t,x) := g_{(\psi_H^1 (\psi_H^t)^{-1})^{-1}}(u(t,x))
\ee
where $g_\psi$ is the \emph{conformal exponent} of the contactomorphism defined by
the equation $\psi^*\lambda = e^{g_\psi}\lambda$ for positive contactomorphisms $\psi$.

In \cite{oh:entanglement1}, \cite{oh-yso:spectral} and \cite{oh:shelukhin-conjecture},
we develop the theory for the study of contact
dynamics which concerns \emph{entanglement} of Legendrian links which relies on
the study of the moduli spaces of Hamiltonian-perturbed \emph{contact instantons
under the Legendrian boundary condition}, and make its application to a quantitative study of contact dynamics.
The heart of the matter lies in the quantitative study of the moduli space of
perturbed contact instantons and its interplay
with the contact Hamiltonian geometry and calculus. In particular, we have 
constructed the Legendrian contact instanton cohomology (abbreviated as Legendrian CI cohomology
from now on) in \cite{oh-yso:spectral} for the 1-jet bundle and in \cite{oh:entanglement1} 
for the case with small energy, i.e., for the case without bubbling.

All necessary analytical foundations, in the more general setting of \emph{perturbed contact
instantons}, such as a priori elliptic estimates, Gromov-Floer-Hofer style convergence, 
relevant Fredholm theory and the generic mapping and evaluation transversalities
are already established in \cite{oh:entanglement1}, \cite{oh:entanglement1},
\cite{oh:contacton-gluing}, \cite{oh:perturbed-contacton} and
\cite{oh:contacton-transversality} which
 will not be repeated in the present paper but
referred thereto. Only the essential components are proved and
others are summarized in Appendix for readers' convenience and for the self-containedness of
the paper. Frequently we will be brief
leaving only the core of the argument
when we mention these analytical details. We also
use the same notations and terminologies therefrom.
(We refer readers to the survey article 
\cite{oh-kim:survey} for some coherent exposition of these
analytic foundations of (perturbed) contact instantons,
 and with some comparison with the framework of
 pseudoholomorphic curves on the symplectization.)

The main goals of our study in the present paper is to continue the study to the general case 
including the bubbles and 
to construct a Fukaya-type category of Legendrian submanifolds. 

The present paper focus on the case of unperturbed contact instantons (i.e., the case of $H = 0$)
but generalizing the domain $\R \times [0,1]$ of the maps $u$ to
the disc-type surface $\dot \Sigma$ with boundary punctures. We postpone the study of the case $H \neq 0$ 
till \cite{kim-oh:long-exact} and others.

 We  follow the off-shell analytical framework of \cite{oh:contacton-transversality} which is in turn an
 adaptation thereof from  \cite{oh:contacton}, \cite{oh-savelyev}  to the current
 boundary value problem 
 \be\label{eq:contacton-Legendrian-bdy}
\begin{cases}
\delbar^\pi u = 0, \quad d(u^*\lambda \circ j) = 0,\\
u(\overline{z_iz_{i+1}}) \subset \mathsf R
\end{cases}
\ee 
where we denote by $\mathsf R = \sqcup R_{j=1}^k$ a Legendrian link following the notation from
\cite{oh-yso:index}.  

We recall the setting of our conventions and basic notations which 
are not the same as those appearing in the literature concerning the Hamiltonian Floer theory 
of symplectic manifolds with contact ideal boundary. These conventions are continuations
of those listed as `List of conventions' in the present author's book \cite{oh:book1} and of \cite{dMV}.

\bigskip

\noindent{\bf Conventions and Terminologies:}

\medskip

\begin{itemize}
\item {(Symplectic Hamiltonian vector field)} 
The Hamiltonian vector field on symplectic manifold $(P, \omega)$ is defined by $X_H \rfloor \omega = dH$.
We denote by
$$
\phi_H: t\mapsto \phi_H^t
$$
its Hamiltonian flow.
\item {(Contact Hamiltonian)} The contact Hamiltonian of a time-dependent contact vector field $X_t$ is
given by
$$
H: = - \lambda(X_t).
$$
We denote by $X_H$ the contact vector field whose associated contact Hamiltonian is given by $H = H(t,x)$, and its flow by $\psi_H^t$.
\item We denote by $R_\lambda$ the Reeb vector field associated to $\lambda$. We denote by
$$
\phi_{R_\lambda}^t
$$
its flow. Then we have $\phi_{R_\lambda}^t  = \psi_{-1}^t$, i.e., \emph{the Reeb flow is the contact Hamiltonian
flow of the constant Hamiltonian $H = -1$}.
\end{itemize}

\bigskip

\noindent{\bf Acknowledgement:} We thank Jongmyeong Kim for some comments on the 
preliminary version of the present paper.

\section{Iso-speed Reeb chords and their nondegeneracy}

We start with recalling the following standard definition.

\begin{defn}\label{defn:spectrum} Let $\lambda$ be a contact form of contact manifold $(M,\xi)$ and $R \subset M$ a
connected Legendrian submanifold.
Denote by $\frak{X}(M,\lambda)$ (resp. $\frak{X}(M,R;\lambda)$) the set of closed Reeb orbits
(resp. the set of self Reeb chords of $R$).
\begin{enumerate}
\item
We define $\operatorname{Spec}(M,\lambda)$ to be the set
$$
\operatorname{Spec}(M,\lambda) = \left\{\int_\gamma \lambda \mid \lambda \in \frak{X}(M,\lambda)\right\}
$$
and call the \emph{action spectrum} of $(M,\lambda)$.
\item We define the \emph{period gap} to be the constant given by
$$
T(M,\lambda): = \inf\left\{\int_\gamma \lambda \mid \lambda \in \frak{X}(M,\lambda)\right\} > 0.
$$
\end{enumerate}
We define $\operatorname{Spec}(M,R;\lambda)$ and the associated $T(M,\lambda;R)$ similarly using the set
$\frak{X}(M,R;\lambda)$ of Reeb chords of $R$.
\end{defn}
We set $T(M,\lambda) = \infty$ (resp. $T(M,\lambda;R) = \infty$) if there is no closed Reeb orbit (resp. no $(R_0,R_1)$-Reeb chord).
Then we define
\be\label{eq:TMR}
T_\lambda(M;R): = \min\{T(M,\lambda), T(M,\lambda;R)\} > 0
\ee
and call it the \emph{(chord) period gap} of $R$ in $M$.

Consider a 2-component link, i.e., a pair of \emph{connected} Legendrian submanifolds $(R, R')$.

\begin{defn}\label{defn:generators} We denote by
$$
\frak{X}(R,R')
$$
the set of nonnegative iso-speed Reeb chords, i.e., the set of pairs $(\gamma, T)$
satisfying
\be\label{eq:chord-equation}
\begin{cases}
\dot \gamma(t) = T R_\lambda(\gamma(t)),\\
\gamma(0) \in R, \quad \gamma(1) \in R'.
\end{cases}
\ee
\end{defn}
We note that there is a natural involution 
$\tau: \frak{X}(R,R') \to \frak{X}(R',R)$ defined by
\be\label{eq:tau}
\tau(\gamma,T): = (\widetilde \gamma, -T)
\ee
where $\widetilde \gamma$ is the time-reversal of $\gamma$ defined by $\widetilde \gamma(t): = \gamma(1-t)$.

It is useful for us to recall the following notion of
the \emph{Reeb trace} introduced in \cite{oh:entanglement1}.

\begin{defn}[Reeb trace] Let $R$ be a Legendrian submanifold of $(M,\lambda)$.
The \emph{Reeb trace} denoted by $Z_R$ is the union
$$
Z_R: = \bigcup_{t \in \R} \phi_{R_\lambda}^t(R).
$$
\end{defn}


The following is proved in \cite[Appendix B]{oh:contacton-transversality}, which is the
relative verison of a theorem in \cite{ABW}.

\begin{thm}[Theorem B.3 \cite{oh:contacton-transversality}]\label{thm:Reeb-chords-lambda-text}
Let $(M,\xi)$ be a contact manifold. Let  $(R_0,R_1)$ be a pair of Legendrian submanifolds
allowing the case $R_0 = R_1$.  There
exists a residual subset $\operatorname{Cont}^{\text{\rm reg}}_1(M,\xi) \subset \CC(M,\xi)$
such that for any $\lambda \in \operatorname{Cont}^{\text{\rm reg}}_1(M,\xi)$ all
Reeb chords from $R_0$ to $R_1$ are nondegenerate for $T > 0$, and
Bott-Morse nondegenerate when $T = 0$.
\end{thm}

We consider the Moore path space as the set
$$
\Theta(M;R,R'): = \left\{ (\gamma,T) \, \Big\vert\, \gamma: [0,1] \to M,\,  \gamma(0) \in R, \, \gamma(1) \in R',
\quad \int \gamma^*\lambda = T \right \}.
$$
Recall in our definition that the set
$\mathfrak{X}(R,R'): = \mathfrak{X}(R,R';\lambda)$ consists of iso-speed Reeb chords, i.e., the
set of pairs  $(\gamma, T) \in \Theta(M;R,R')$ satisfying the equation
$$
\dot \gamma = T R(\lambda), \quad T = \int \gamma^*\lambda.
$$
We consider the assignment
\be\label{eq:Phi}
\Phi: (T,\gamma,\lambda) \mapsto \dot \gamma - T \,R_\lambda(\gamma)
\ee
as a section of the  vector bundle over
$$
(0,\infty) \times \Theta(M;R,R') \times \Cont(M,\xi).
$$
 We have
$$
\dot \gamma - T\, R_\lambda(\gamma)
\in \Gamma(\gamma^*TM; T_{\gamma(0)}R_0, T_{\gamma(1)}R_1).
$$
We  define the vector bundle
$$
\CH(R_0,R_1) \to (0,\infty) \times \Theta(M;R_0,R_1) \times \Cont(M,\xi)
$$
whose fiber at $(T,\gamma,\lambda)$ is $\Gamma(\gamma^*TM)$. We denote by
$\pi_i$, $i=1,\, 2, \, 3$ the corresponding projections as before.

We denote $\frak{X}(M,\lambda;R_0,R_1) = \Phi_\lambda^{-1}(0)$,
where
$$
\Phi_\lambda: = \Phi|_{ (0,\infty) \times \Theta(M;R_0,R_1) \times \{\lambda\}}.
$$
Then we have
$$
\mathfrak{X}(R_0,R_1;\lambda) =  \Phi_\lambda^{-1}(0) = \frak{X}(M,\xi) \cap \pi_3^{-1}(\lambda).
$$
The proof of Theorem \ref{thm:Reeb-chords-lambda-text}
relies on the following proposition which is also
 proved in \cite[Appendix B]{oh:contacton-transversality}.

\begin{prop} [Proposition B.4 \cite{oh:contacton-transversality}]
\label{prop:nondeneracy-chords-text}
Suppose $R_0 \cap R_1 = \emptyset$.
A Reeb chord $(\gamma, T)$ of $(R_0,R_1)$ is nondegenerate if and only if
the linearization
$$
d_{(\gamma, T)}\Phi: \R \times \Gamma\left(\gamma^*TM;T_{\gamma(0)}R_0,T_{\gamma(1)}R_1\right)
\to \Gamma(\gamma^*TM)
$$
is surjective.
\end{prop}

Now we consider a \emph{general Legendrian link} $\vec R = (R_1, \ldots, R_k)$ be given. 
We consider a boundary marked bordered Riemann surface 
$
(\Sigma, \{z_1, \cdots, z_m\})
$
and decompose the boundary of $\Sigma$ into
$$
\del \Sigma = \sqcup_{j=1}^h \del_j \Sigma
$$
the union of connected components $\del_j \Sigma \cong S^1$.
We then take its associated  punctured bordered Riemann surface 
$$
\dot \Sigma: = \Sigma \setminus \{z_1, \cdots, z_m\}.
$$
\begin{defn}\label{defn:nondegeneracy-chords}
We say a Reeb chord $(\gamma, T)$ of $(R_0,R_1)$ is nondegenerate if
the linearization map $\Psi_\gamma = d\phi^T(p): \xi_{\gamma(0)} \to \xi_{\gamma(1)}$ satisfies
$$
\Psi_\gamma(T_{\gamma(0)} R_0) \pitchfork T_{\gamma(1)} R_1  \quad \text{\rm in }  \,  \xi_{\gamma(1)}
$$
or equivalently
$$
\Psi_\gamma(T_{\gamma(0)} R_0) \pitchfork T_{\gamma(1)} Z_{R_1} \quad \text{\rm in} \, T_{\gamma(1)}M.
$$
\end{defn}
More generally, we consider the following situation.
We recall the definition of \emph{Reeb trace} $Z_R$ of a Legendrian submanifold
$$
Z_R: = \bigcup_{t \in \R} \phi_{R_\lambda}^t(R).
$$
(See \cite[Appendix B]{oh:contacton-transversality} for detailed discussion on its genericity.)
\begin{hypo}[Nondegeneracy]\label{hypo:nondegeneracy}
Let $\vec{R}=(R_1,\cdots,R_k)$ be a chain of Legendrian submanifolds.
We assume that we have
$$
Z_{R_i} \pitchfork R_j
$$
for all $i \neq  j= 1,\ldots, k$.
\end{hypo}
Here $\phi^t_{R_\lambda}$ is the flow generated by the Reeb vector field $R_\lambda$.
We note that this nondegeneracy is equivalent the nondegeneracy of Reeb chords between $R_i$ and $R_j$.
(See \cite[Corollary 4.8]{oh:entanglement1}.)

\section{Bridged Legendrian links and grading of Reeb chords}
\label{sec:anchored}

In this section and the next, we provide the off-shell framework of
the contact instantons attached to a Legendrian link in general by closely following the exposition of 
\cite[Par 2]{oh-yso:index}.

Let $\vec R = \{R_1, \cdots, R_k\}$ be a Legendrian link consisting
of connected Legendrian submanifolds $R_i$. We denote by
$$
\mathsf R = \bigsqcup_{i=1}^k R_i
$$
its disjoint union.

\begin{defn}\label{defn:Reeb-chords-intro} 
Let $(\gamma, T)$ be an iso-speed Reeb chord with $\gamma:[0,1] \to M$ satisfying
\be\label{eq:isospeed-Reeb}
\begin{cases}
\dot \gamma = T R_\lambda(\gamma(t)),\\
\gamma(0), \, \gamma(1) \in \mathsf R.
\end{cases}
\ee
 We denote by
$
\mathfrak{X}(M,\vec R;\lambda) 
$
the set of iso-speed $\lambda$-Reeb chords of the Legendrian link $\vec R$.
\end{defn}

We have the decomposition
\be\label{eq:Reeb-decompose}
\mathfrak{X}(M,\vec R;\lambda) = \bigcup_{(i,j) \in \underline{k} \times \underline{k}} 
\mathfrak{X}(R_i,R_j;\lambda).
\ee
By definition, each Reeb chord $\gamma$ satisfies
$$
\gamma(0) \in R_i, \quad \gamma(1) \in R_j
$$
respectively for some pair $i, \, j \in \underline k$ allowing the case $i = j$.

\begin{defn}\label{defn:gamma01} We
have two maps $s, \, t: \frak{X}(M,\vec R;\lambda) \to \{1, \ldots, k\}=:\underline{k}$ the values of which
are defined to be   $s(\gamma): = i, \quad t(\gamma) : = j$
when $\gamma$ satisfies 
\be\label{eq:gamma01}
\gamma(0) \in R_i, \quad \gamma(1) \in R_j.
\ee
We call the associated components the \emph{source} and the \emph{target} component
respectively.
\end{defn}

In terms of this definition, we can rephrase the definitions of self-chords and trans-chords 
from \cite{oh:entanglement1} as follows.

\begin{defn} We say an iso-speed Reeb chord $(\gamma,T)$ a \emph{self-chord} 
if $s(\gamma) = t(\gamma)$
and a \emph{trans-chord} otherwise.
\end{defn}
With this definition, we have another decomposition
\be\label{eq:Reeb-decompose-self-trans}
\mathfrak{X}(M,\vec R;\lambda) = \mathfrak{X}^{\text{\rm self}}(M,\vec R;\lambda)
\cup \mathfrak{X}^{\text{\rm trans}}(M,\vec R;\lambda)
\ee
where we have further decompositions
$$
\mathfrak{X}^{\text{\rm self}}(M,\vec R;\lambda) =  \bigcup_{i \in \underline{k}} 
\mathfrak{X}(R_i,R_i;\lambda)
$$
and
$$
\mathfrak{X}^{\text{\rm trans}}(M,\vec R;\lambda) = 
\bigcup_{(i,j) \in \underline{k} \times \underline{k};i \neq j} 
\mathfrak{X}(R_i,R_j;\lambda).
$$
Then by definition, we have
$$
\gamma \in \mathfrak{X}\left(R_{s(\gamma)}, R_{t(\gamma)}\right)
$$
and the decomposition
$$
\vec \gamma = \vec \gamma^{\text{\rm self}} \sqcup \vec \gamma^{\text{\rm trans}}.
$$
For the purpose of computing the Fredholm index of the linearized operator in terms of a topological index, 
we utilize the notion of \emph{bridged Legendrian links} which is an adaptation of that of based Lagrangian pairs studied in \cite{fooo:book1} in symplectic geometry.

We denote by $\Lambda_{(\xi,d\lambda)}= \Lambda_{(\xi,d\lambda)}(M)  \to M$ the bundle of Lagrangian Grassmannians 
of the symplectic vector bundle  $(\xi, d\lambda) \to M$. Its fiber at $x \in M$ is given by 
the set
$$
\Lambda_{(\xi,d\lambda)}|_x = Lag(\xi_x,d\lambda|_x)
$$
of Lagrangian subspaces of  the symplectic vector space $(\xi_x, d\lambda|_x)$.

\begin{defn}[Graded bridges]\label{defn:graded-bridges}
Let $(M,\xi)$ be a contact manifold equipped with a contact form $\lambda$.
\begin{enumerate}
\item Let $(R,R')$ be a pair of
Legendrian submanifolds and $\ell$ be a base path from $R$ to $R'$, which we call a \emph{bridge}
from $R$ to $R'$. A \emph{grading} of $\ell$ is a section of $\ell^*\Lambda_{(\xi,d\lambda)}$ and denote by
$\alpha$.
\item A \emph{graded system of bridges}
is a collection of graded bridges $(\ell_{ij},\alpha_{\ij})$ for each given
pair $R_i, \, R_j \in \vec R$. We denote by $\vec \ell$ the collection $\{\ell_{ij}\}$.
\item A Legendrian link $\vec R = (R_1, \cdots, R_k)$ equipped with a graded  system of bridges
is said to be \emph{graded bridged}.
\end{enumerate}
\end{defn}

\section{Off-shell bordered maps and enumeration of asymptotic Reeb chords}

Next let $(\Sigma, \{z_1, \cdots, z_m\})$ be a compact bordered compact Riemann surface 
with boundary $\del \Sigma$ decomposed into the union
$$
\del \Sigma = \sqcup_{j=1}^h \del_j\Sigma
$$
of its connected components $\del_j \Sigma \cong S^1$. We assume 
 $h \geq 1$ for the number $h$ of boundary components in the present paper.

Then we decompose the set $\vec z$ into the union of
the boundary marked points $\vec z = \sqcup_{j=1}^h \vec z^j$:
For each $j$, we enumerate the subset $\vec z^j$ counterclockwise (with respect to the boundary
orientation of $\del \Sigma$) into 
$$
\vec z^j = \{z^j_1,\ldots, z^j_{s_j}\} \subset \del_j \Sigma
$$
which partitions $\del_j \Sigma$ into the union of arc-segments
\be\label{eq:ejalpha}
\del_j \Sigma = \cup_{\alpha=1}^{s_j} e^j_\alpha, \, \quad e^j_\alpha: = \overline{z^j_\alpha z^j_{\alpha+1}}.
\ee
Then we have
\be\label{eq:m-sum-ellj}
m = \sum_{j=1}^h s_j.
\ee

We have another enumeration 
\be\label{eq:vec-zj}
\vec z^j = \left\{p^+_{j,1},\ldots,p^+_{j,s^+_j},p^-_{j,1},\ldots,p^-_{j,s^-_j}\right\}, \quad s^+_j + s^-_j= s_j
\ee
of the elements of $\vec z^j$ consisting of those of postive punctures and of negative ones.
Then the (boundary) punctured surface
$$
\dot \Sigma = \Sigma \setminus \vec z 
$$
carries the strip-like ends near the punctures on some pairwise disjoint open subsets thereof
\beastar
& &U_{i,\alpha}^+ = \phi_{i,\alpha}^+([0,\infty)\times[0,1])\subset \dot \Sigma \quad {\rm for} \; \alpha=1,\ldots,s^+_j \\
& &U_{j,\beta}^- = \phi_{j,\beta}^-((-\infty,0]\times[0,1])\subset \dot \Sigma \quad {\rm for} \; \beta =1,\ldots,s^-_j
\eeastar
where $\phi_{i,\alpha}^+$'s and $\phi_{j,\beta}^-$'s are bi-holomorphic maps from the corresponding semi-strips 
into $\dot \Sigma$ respectively. We also put
$$
s^+ = \sum_{j = 1}^h s^+_j, \quad s^- = \sum_{j=1}^h s^-_j
$$
which are further decomposed into 
$$
s^\pm = s^\pm_{\text{\rm self}} + s^\pm_{\text{\rm trans}}
$$
by now in the obvious way. Then we have 
\be\label{eq:m-sum-s=-}
m = s^+ + s^- = (s^+_{\text{\rm self}} + s^+_{\text{\rm trans}})
 + (s^-_{\text{\rm self}} + s^-_{\text{\rm trans}}).
\ee

\subsection{Asymptotic boundary conditions at the punctures}

Once the above identification of asymptotic boundary conditions for the linearized operators
is established,
we are ready to provide the precise off-shell function spaces of the moduli space of
contact instantons. 
We will be brief in their off-shell description by referring readers
\cite[Section 11.2]{oh:contacton} for the closed case which can be 
duplicated with incorporation of boundary conditions. 

We consider any smooth map $u:\dot \Sigma \to M$,  not necessarily being a contact instanton,
which satisfies  the boundary condition
\be\label{eq:contacton-Legendrian-bdy}
u(e^j_\alpha) \subset R_j,\quad e^j_\alpha: = \overline{z^j_\alpha z^j_{\alpha+1}}
\ee
where $e^j_\alpha$ is the arc-segment of $\del \dot \Sigma$ given in \eqref{eq:ejalpha}.

Under the nondegeneracy conditions given in Hypothesis \ref{hypo:nondegeneracy}, 
there exist two collections of Reeb chords
\beastar
\underline \gamma_i & = & \{\gamma_{i,1}^+,\cdots, \gamma^+_{i,s^+_j}\} \\
\overline \gamma_j & = & \{\gamma_{j,1}^-,\cdots, \gamma^-_{j,s^-_j}\}
\eeastar
for $i,\, j = 1, \ldots, h$ at the positive and at the negative punctures of $\del \dot \Sigma$ respectively 
such that $u$ satisfies
\be\label{eq:u-bdy-condition}
u(z) \in R_i \quad \text{ for } \, z \in e^i_\alpha \subset \del_i \dot \Sigma.
\ee
Furthermore at the punctures $p^+_{i,\alpha}$ and $p^-_{j,\beta}$, $u$ satisfies
 the asymptotic condition 
\be\label{eq:asymp-condition}
\begin{cases}
\lim_{\tau \to \infty}u((\tau,t)_{i,\alpha}) = \gamma^+_{i,\alpha}(T_{i,\alpha}(t+t_{i,\alpha})), \\
\lim_{\tau \to - \infty}u((\tau,t)_{j,\beta}) = \gamma^-_{j,\beta}(T_{j,\beta}(t-t_{j,\beta}))
\end{cases}
\ee
for some $\gamma^+_{i,\alpha}$ and $\gamma^-_{j,\beta}$ and
$t_{i,\alpha}, \, t_{j,\beta} \in \del_j\Sigma$ respectively.
Here $t_{i,\alpha},\, t_{j,\beta}$ depend on the given strip-like coordinate, and
$$
T_{i,\alpha} = \int (\gamma^+_{i,\alpha})^*\lambda, \, 
T_{j,\beta} = \int ( \gamma^-_{j,\beta})^*\lambda
$$
are the periods of relevant Reeb chords.

\begin{defn}[Off-shell bordered maps]\label{defn:offshell-borderedmap}
We call a map satisfying the boundary condition 
\eqref{eq:u-bdy-condition}, \eqref{eq:asymp-condition} an \emph{off-shell bordered map
attached to the Legendrian link $\vec R$}.
\end{defn}  

\subsection{Symplectic bundle pair associated to asymptotic Reeb chords}

Let $(R, R'; \ell)$ be a 2-component bridged Legendrian link  that satisfies
transversality stated in Hypothesis \ref{hypo:nondegeneracy}, i.e.,
$Z_{R} \pitchfork R'$, and consider its grading $\alpha$
which is a section of $\ell^*\Lambda_{2n}(\xi,d\lambda) \to [0,1]$.
Then we are given a symplectic bundle pair 
$$
(\ell^*\xi, \alpha) \to ([0,1], \del [0,1]).
$$ 
Let $\gamma$ be an iso-speed Reeb chord  from $R$ to $R'$ and consider a map 
$w:[0,1]^2 \rightarrow M$ satisfying the boundary condition
\bea
& &w(\tau,0) \in  R, \; w(\tau,1)\in R', \label{eq:bdy-R0R1}\\
& &w(0,t) =  \ell(t), \; w(1,t)=\gamma(t). \label{eq:asymp-gamma}
\eea
Denote by $[w,\gamma]$ the homotopy class of such $w$'s.

\begin{rem}\label{rem:[w,gamma]} By the exponential convergence result from \cite{oh-wang:CR-map1},
\cite{oh-yso:index},
we can compactify the contact instanton map $w: [0,\infty) \times [0,1] \to M$
satisfying
$$
\begin{cases}
w(0,t)=\ell(t), \; w(\tau,0)\in R, \\
 w(\tau,1)\in R', \; \lim_{\tau\rightarrow\infty}w(\tau,t) = \gamma(t)
 \end{cases}
$$
to a continuous map from $\overline w: [0,\infty) \cup \{\infty\} \times [0,1] \cong [0,1]^2$. This latter map
defines a  relative homotopy class $[w,\gamma]$ as that of a map $\overline w$
where the condition at $\tau =1$ is replaced by the
asymptotic condition at $\tau =\infty$. We denote by
$$
[w,\gamma]
$$
the homotopy class of this map $\overline w$. The analytical index of the contact instanton $w$ will
be expressed in terms of some topological index of this compactified map $\overline w$.
(See \cite{fooo:anchored} for a similar practice.)
\end{rem}
Now we are ready to assign an integer grading to each Reeb chord $\gamma \in \frak{X}(R,R')$ 
when the link $(R, R')$ is graded bridged by $(\ell, \alpha)$.
More specifically we consider an iso-speed Reeb chord $(\gamma, T_0)$ with
$T_0 = \int \gamma^*\lambda$. Then we can express $\gamma$ as 
$\gamma(t)=\phi^{T_0t}_{R_\lambda}(\gamma(0))$. Note that $T_0$
can be either positive or negative. 
If $T_0 < 0$, then the curve $\gamma(t)$ is tangent to the Reeb vector field but
in the opposite direction. 
Note that $d\phi^{T_0}_{R_\lambda}(T_{\gamma(0)}R)$ and $T_{w(\tau,1)}R'$
are transversal to each other by the nondegeneracy hypothesis given in
Hypothesis \ref{hypo:nondegeneracy}.

Now let $w: [0,1]^2 \to M$ be a smooth map satisfying
\eqref{eq:bdy-R0R1} and \eqref{eq:asymp-gamma}, and
$$
 (\ell, \alpha_\ell)
$$ 
be a graded bridge from $R$ to $R'$. Next for given Reeb chord $\gamma \in \mathfrak{X}(R,R')$
let $\alpha_\gamma$ be the section of $\gamma^*\Lambda_{(\xi,d\lambda)}$ given by
$$
\alpha_\gamma(t) = d\phi^{T_0t}_{R_\lambda}\left(T_{\gamma(0)}R\right).
$$
Then we pick a Lagrangian path  $\alpha^-$
from $T_{w(\tau,1)}R'$ to $d\phi^{T_0}_{R_\lambda}(T_{\gamma(0)}R)$  that satisfy the conditions
\begin{itemize}
\item $\alpha^-(0) = T_{w(\tau,1)}R', \, \alpha^-(1) = d\phi^{T_0}_{R_\lambda}(T_{\gamma(0)}R)$,
\item $\alpha^-(t) \in Lag(\xi_p,d\lambda|_p)\setminus Lag_1(\xi_p, d\lambda|_p; T_pR')$ for $p = \gamma(0)$,
\item $-(\alpha^-)'(1)$ is positively directed.
\end{itemize}
as in \cite[Proposition 8.3]{oh-yso:index}. 
Concatenating these 3 paths with the Gauss maps 
$$
\tau\in [0,1] \mapsto  TR|_{w(\tau,0)}, \,TR'|_{w(1-\tau,1)},
$$
we have a continuous section over $\del [0,1]^2$ 
\be\label{eq:tilde-alpha}
\widetilde{\alpha}_{([w,\gamma];\alpha_\ell)} := (\alpha_{\widetilde \ell}) \#( TR|_{w(\tau,0)})
\# (\alpha_\gamma)\#(\alpha^-) \# (TR'|_{w(1-\tau,1)}).
\ee
\begin{rem}\label{rem:pentagon}
Strictly speaking, $\widetilde{\alpha}_{([w,\gamma];\alpha_\ell)}$ is defined
 over the `pentagon' made by replacing the $(1,1)$ vertex 
of $\del [0,1]^2$ by an interval over which $\alpha^-$ is put. For the simplicity of exposition, we will
keep saying `over $\del [0,1]^2$' here and henceforth.
\end{rem}

This way we obtain the following symplectic bundle pair.

\begin{defn}[Symplectic bundle pair over {$[0,1]^2$}] \label{defn:symp-bundle-pair}
Let $w: [0,1]^2 \to M$ be a smooth map satisfying
\eqref{eq:bdy-R0R1} and \eqref{eq:asymp-gamma}, and $\alpha$ be a grading of
$\ell$. Then it induces the symplectic bundle pair
\be\label{eq:bundlepair-square}
\left(w^*\xi, \widetilde \alpha_{([w,\gamma];\alpha_\ell)}\right) \to ([0,1]^2, \del [0,1]^2).
\ee
\end{defn}

\subsection{Absolute grading of Reeb chords}

Now we are ready to associate 
a natural grading of the Reeb chord $\gamma$ by assigning the Maslov index 
of this symplectic bundle pair.

Recalling the decomposition \eqref{eq:Y-decompose} of 
 $TM = \xi \oplus \R \langle R_\lambda \rangle$,
we choose a trivialization $\Phi:w^*TM \rightarrow [0,1]^2\times\R^{2n+1}$
by taking the diagonal form of the trivialization as follows.

\begin{defn}[Diagonal trivialization $\Phi$ of $w^*TM$]\label{defn:Phi} $\Phi$ satisfies
\begin{itemize}
\item it respects the splitting $w^*TM = w^*\xi \oplus \R\langle R_\lambda \rangle$ and
$$
\R^{2n+1} \cong \C^n \oplus \R,
$$
\item  $\Phi|_{w^*\xi}:w^*\xi \rightarrow [0,1]\times\R^{2n}\times\{0\} \cong [0,1]^2\times\R^{2n}$ is a symplectic trivialization.
\end{itemize}
\end{defn}

Then we can associate a Lagrangian path $\alpha^\Phi_{[w,\gamma];\alpha}$ 
defined on $\partial[0,1]^2\setminus \{(1,1)\}$ by
\be\label{eq:lag-loop-piece}
\begin{cases}
\alpha^\Phi_{[w,\gamma];\alpha}(\tau,i)  =  \Phi(T_{w(\tau,i)}R_i) \qquad {\rm for}
\; i=0,1  \\
\alpha^\Phi_{[w,\gamma];\alpha}(0,t)  =  \Phi(\alpha(t)), \\
 \alpha^\Phi_{[w,\gamma];\alpha}(1,t)=\Phi(d\phi^{T_0t}_{R_\lambda}(T_{\gamma(0)}R_0)).
\end{cases}
\ee
 Note that
$$
\lim_{t\rightarrow 1}\alpha^\Phi_{[w,\gamma];\alpha}(1,t)=\Phi(d\phi^{T_0}_{R_\lambda}(T_{\gamma(0)}R_0))
$$
and
$$
\lim_{\tau\rightarrow 1}\alpha^\Phi_{[w,\gamma];\alpha}(\tau,1)=\Phi(T_{\gamma(1)}R_1).
$$
\begin{defn}[Grading of Reeb chords]\label{defn:grading-chord}
We define $\mu([w,\gamma];\alpha)$ to be the Maslov index of this Lagrangian loop 
$\widetilde{\alpha}^\Phi_{[w,\gamma];\alpha}$ in $(\R^{2n},\omega_0)$.
We call $\mu([w,\gamma];\alpha)$  the Maslov index of the bundle pair
\eqref{eq:bundlepair-square}.
\end{defn}
Note that this index is independent of the choice of the diagonal trivializations $\Phi$
and depends only on the $[w,\gamma]$ homotopy classes.

\begin{rem}\label{rem:GammaRR'}
In general, this defines a $\Z/\Gamma$ grading where $\Gamma = \Gamma_{RR'}$ is the positive generator of the
abelian group
\be\label{eq:Gamma}
\Gamma_{RR'}: = \{ \mu(A) \in \Z \mid A \in \pi_1(\Omega(R,R'))\}
\ee
where $\mu(A)$ is the Maslov index of the annulus map $v: S^1 \times [0,1] \to M$
with $[v] = A$ satisfying the boundary condition
$$
v(\theta,0) \in R, \, v(\theta,1) \in R' \quad \theta \in S^1.
$$
\end{rem}

\section{Off-shell energy of contact instantons on disc-type surfaces}
\label{sec:offshellenergy}

We restrict ourselves to the case of $g=0$ in the present paper, i.e., $(\dot \Sigma,j)$ 
of disc-type surfaces.
We first recall the following basic existence and the uniqueness result 
on the minimal area metrics proved in \cite{zwiebach,wolf-zwieb}.

\begin{thm}[Zwiebach] When $g = 0$ and $k \geq 3$, there exists a unique
minimal area metric associated to each $(\Sigma, j) \in \CM_{0,k}$, which continuously extends to
the compactification $\overline \CM_{0,k}$.
\end{thm}
In other words, the minimal area metric provides a natural slice to the well-known
isomorphism between 
$$
\{\text{\rm complex structures}\} \longleftrightarrow  \{\text{\rm conformal isomorphism classes of
associated metrics}\}
$$
and respects the sewing rule of the degeneration of conformal structures.
A representation of the conformal structure on the boundary punctured discs,
was used in Fukaya and the author's work \cite{fukaya-oh} in their study of
adiabatic degeneration of pseudo-holomorphic polygons with Lagrangian
boundaries on the cotangent bundle.

\subsection{Definition of off-shell energy}
\label{subsec:definition-energy}

We  consider a smooth map $w: (\dot \Sigma, \del \dot \Sigma) \to (M, \mathsf R)$.
The norm $|dw|$ of the map
$$
dw:(T\Sigma,h) \to (TM, g)
$$
with respect to the metric $g$ is defined by
$$
|dw|_g^2 := \sum_{i=1}^{2} {|dw(e_i)|_g}^2,
$$
where  $\{ e_1, e_2 \}$ is an orthonormal frame of $T \Sigma$
with respect to $h$.

The following are the consequences from the definition of
contact Cauchy-Riemann map and the compatibility of $J$ to $d\lambda$ on $\xi$, whose
proofs we omit but refer to \cite{oh-wang:CR-map1}.

\begin{prop}\label{prop:energy-omegaarea}
Denote $g_J=\omega(\cdot, J \cdot)$ and the associated norm
by $|\cdot| = |\cdot|_J$. Fix a Hermitian metric $h$ of $(\Sigma,j)$,
and consider a smooth  map $u:\Sigma \to M$. Then we have
\begin{enumerate}
\item[(1)] $|d^\pi w|^2 = |\del^\pi w| ^2 + |\delbar^\pi w|^2$,
\item[(2)] $2w^*d\lambda = (-|\delbar^\pi w|^2 + |\del^\pi w|^2) \,dA $
where $dA$ is the area form of the metric $h$ on $\Sigma$.
\item[(3)] $w^*\lambda \wedge w^*\lambda \circ j = |w^*\lambda|^2\, dA$
\item[(4)] $|\nabla (w^*\lambda)|^2 = |dw^*\lambda|^2 + |\delta w^*\lambda|^2$.
\end{enumerate}
\end{prop}

We then introduce the $\xi$-component of the harmonic energy, which we call
the \emph{$\pi$-harmonic energy}. 

\begin{defn}\label{defn:pi-energy}
For a smooth map $\dot \Sigma \to M$, we define the $\pi$-energy of $w$ by
\be\label{eq:Epi}
E^\pi(j,w) = \frac{1}{2} \int_{\dot \Sigma} |d^\pi w|^2.
\ee
\end{defn}
We have the equality 
\be\label{eq:contact-area}
E^\pi(j,w)  = \int_{\dot \Sigma} w^*d\lambda 
\ee
`on shell' i.e., for any contact Cauchy-Riemann map, which satisfies $\delbar^\pi w = 0$.

As discovered by Hofer in \cite{{hofer:invent}} in the context of symplectization, one
needs to examine the $R_\lambda$-part of energy that controls the asymptotic behavior of
contact instantons near the puncture. For this purpose, the \cite {hofer:invent}-type energy
introduced in \cite{{hofer:invent}} is crucial. In this section, we generalize this energy
to the general context of non-exact case without involving the symplectization.

Following \cite{{hofer:invent}}, \cite{behwz} and \cite{oh:contacton},
 we introduce the following class of
test functions
\begin{defn}\label{defn:CC} We define
\be
\CC =\{\varphi: \R \to \R_{\geq 0} \mid \supp \varphi \, \text{is compact}, \, \int_\R \varphi = 1\}.
\ee
\end{defn}

For the purpose of emphasizing the charge vanishing for the open-string context
established in \cite{oh-yso:index} and the current circumstance of construction of an $A_\infty$ category
for which we consider the domain $\dot \Sigma$ of \emph{disc-type surface} which is
simply connected,  we first verbatim recall the definition of $\lambda$-energy introduced from \cite{oh:contacton} for the closed string context, which we apply to the open string context 
without much change in general. (See Section 4 \& 5 \cite{oh:contacton} for the details.)

This enables us to introduce the following.

\begin{defn}[$\lambda$-energy at a puncture $p$]
We denote the common value of $E_{\CC,f}(j,w)$ by $E^\lambda_p(w)$, and call the \emph{$\lambda$-energy at $p$}.
\end{defn}

On the other hand, by the simply connectedness of $\dot \Sigma$, we can write
$$
w^*\lambda \circ j = dg,
$$
for some globally defined function $g: \cdot \Sigma \to \R$, 
i.e., $w^*\lambda\circ j$ becomes an exact one-form on $\dot \Sigma$. In terms of the strip-like coordinates
$(\tau,t)$ near a puncture, we can express
$g = g(\tau,t)$ and then we have
$$
g = f_p + c_p
$$
for some constant $c_p$. 

\begin{rem} We note that one can lift any smooth exact contact instanton $w: \dot \Sigma \to M$
can be lifted to the symplectization $M \times \R$ by setting $u: = (w,f)$ for each individual instanton $w$
modulo a shift to the radial direction of $M \times \R$.
\end{rem}

\begin{defn}[$E_\CC$-energy] Let $w$ satisfy $d(w^*\lambda \circ j) = 0$. Then we define
$$
E_{\CC}(j,w) = \sup_{\varphi \in \CC} \int_\Sigma d(\psi(g)) \wedge dg\circ j
= \sup_{\varphi \in \CC} \int_\Sigma d(\psi(g)) \wedge (- w^*\lambda).
$$
\end{defn}
 We note that
$$
d(\psi(g)) \wedge dg \circ j = \psi'(g) dg \wedge dg\circ j= \varphi(g) dg \wedge dg\circ j \geq 0
$$
and hence we can rewrite $E_{\CC}(j,w)$ into
$$
E_{\CC}(j,w) = \sup_{\varphi \in \CC} \int_\Sigma \varphi(g) dg \wedge dg\circ j.
$$
For a given smooth map $w$ satisfying $d(w^*\lambda \circ j) = 0$,
we have $E_{\CC;g}(w) = E_{\CC,g}(w)$ whenever $dg = w^*\lambda\circ j\, dt = dg$
on $D^2_\delta(p) \setminus \{p\}$ (and so $g(z) = g(z) + c$ for some constant $c$

This function $g$ seems to deserve a name.

\begin{defn}[Contact instanton potential] We call the above normalized function $g$
the \emph{contact instanton potential} of the contact instanton charge form $w^*\lambda \circ j$.
\end{defn}

The following then is the definition of the total energy.

\begin{defn}[Total energy] Let $w:\dot \Sigma \to Q$ be any smooth map.
We define the total energy of $w$ by
\be\label{eq:total-energy-sum}
E(j,w) = E^\pi(j,w) + E^\lambda(j,w).
\ee
\end{defn}

If $\dot \Sigma$ carries only one puncture, $\dot \Sigma \cong \C$ and so cannot
carry the above minimal area representation but in this case
the closed form $w^*\lambda \circ j$ is automatically exact.
Therefore there exists a function $g: \dot \Sigma \to \R$ such that $w^*\lambda \circ j = dg$
in which case we may regard the pair $(w,g)$ as a pseudoholomorphic map to the
symplectization.

We define
\be\label{eq:finallambdaenergy}
E^\lambda(j,w) = \sup_{\varphi \in \CC} \int_\Sigma d(\varphi(g)) \wedge dg\circ j.
\ee
This energy will be used in our construction of the compactification of moduli space of contact
instantons of genus 0 in a sequel.
Now we define the final form of the off-shell energy.
Let $w:\dot \Sigma \to Q$ be any smooth map.
We define the total energy of $w$ by
\be\label{eq:total-energy}
E(j,w) = E^\pi(j,w) + E^\lambda(j,w)
\ee
In the rest of the paper, we suppress $j$ from the arguments of
the energy $E(j,w)$ and just write $E(w)$ for the simplicity of notations unless necessary 
to emphasize the dependence on $j$.

The following lemma was proved in \cite{oh:contacton} for the closed-string context
whose proof also applies to the open-string context without change.

\begin{lem}[Lemma 8.1 \cite{oh:contacton}]\label{lem:|dw|<infty} 
Suppose $E(w) = E^\pi(w) + E^\lambda(w) < \infty$. Then
$$
\|dw\|_{C^0} < \infty.
$$
\end{lem}

\subsection{Properness of contact instanton potential function and $\lambda$-energy}
\label{subsec:enerbybound}

In this subsection, we recall the explanation in the closed-string
context of arbitrary genus from \cite[Section 9]{oh:contacton} on
the relationship between the $\pi$-energy,
the $\lambda$-energy and the contact instanton potential function $f$, and 
adapt it to the current open-string context of genus 0. 

Let $h$ be the minimal area metric of $(\dot \Sigma,j)$ and 
$w$ and $g$ as above. 
We first note that the function  $g: \dot \Sigma \to \R$ is proper if and only if
\be\label{eq:f-proper}
g(v_j) = \pm \infty
\ee
for all exterior vertex $v_j \in V(T)$ where $T$ is the tree 
associated the minimal area metric of the puncture Riemann surface $(\dot \Sigma,j)$.

We now collect a series results proved in \cite{oh:contacton} 
in the closed-string context whose proofs equally apply without change
except slight adaptations to the current open-string context.

\begin{cor}[Corollary 9.1 \cite{oh:contacton}]
Suppose that $E(w) < \infty$ and let $f$ be the function defined
in section \ref{sec:offshellenergy}. Then $|dg|_{C^0} < \infty$.
\end{cor}

\begin{prop}[Proposition 9.2 \cite{oh:contacton}]\label{prop:proper-energy} 
Suppose that $E^\pi(w) < \infty$ and the function $g:\dot \Sigma \to \R$ is proper.
Then $E(w) < \infty$.
\end{prop}

\begin{prop}[Lemma 9.3 \cite{oh:contacton}]
 Suppose $E^\pi(w) < \infty$ and $f$ is proper. Denote by
$\gamma^+_1, \cdots, \gamma^+_k$ (resp. $\gamma^-_1,\cdots, \gamma^-_\ell$)
the Reeb chords of $R_\lambda$ asymptotic to the positive (resp. negative punctures) of
$\dot \Sigma$. Then
\beastar
E^\pi(w) & = & \sum_{j=1}^k \int \overline\gamma_j^*\lambda - \sum_{i=1}^\ell \int \underline\gamma_i^*\lambda\\
E^\lambda(w) & = & \sum_{j=1}^k \int \overline\gamma_j^*\lambda \\
E(w) & = & 2 \sum_{j=1}^k \int \overline\gamma_j^*\lambda - \sum_{i=1}^\ell \int \underline\gamma_i^*\lambda.
\eeastar
\end{prop}

\section{Tame contact manifolds and tame Legendrian submanifolds}
\label{sec:tameness}

We first recall the definition of the notion of \emph{tame contact manifolds} from \cite{oh:entanglement1}.

\begin{defn}[Definition 1.8 \cite{oh:entanglement1}] Let $(M,\lambda)$ be contact manifold.
A function $\psi: M \to \R$ is called \emph{$\lambda$-tame} at infinity if $\CL_{R_\lambda} d\psi = 0$
on $M \setminus K$ for a compact subset $K$.
\end{defn}
We will mostly just say `tame' omitting `at infinity' unless there is a need
to emphasize the latter.

\begin{defn}[Contact {$J$} quasi-pseudoconvexity]
Let $J$ be a $\lambda$-adapted CR almost complex structure. We call a function $\varphi: M \to \R$
\emph{contact $J$ quasi-plurisubharmonic on $U$}
\bea
-d(d \varphi \circ J) + k d\varphi \wedge \lambda & \geq & 0 \quad \text{\rm on $\xi$}, \label{eq:quasi-pseuodconvex-intro}\\
R_\lambda \rfloor d(d\varphi \circ J) & = & g \, d\varphi \label{eq:Reeb-flat-intro}
\eea
for some functions $k, \, g$ on $U$. We call such a pair $(\varphi,J)$
a \emph{contact quasi-pseudoconvex pair} on $U$.
\end{defn}
These properties of $J$ enable us to apply the maximum principle
 in the study of the analysis of the moduli space
of contact instantons associated to $J$.  We refer readers to \cite{oh:entanglement1} for more detailed discussion on the
analysis of contact instantons on tame contact manifolds.

\begin{defn}[Tame contact manifolds, Definition 1.12 \cite{oh:entanglement1}]\label{defn:tame}
Let $(M,\xi)$ be a contact manifold, and let $\lambda$ be a contact form of $\xi$.
\begin{enumerate}
\item  We say $\lambda$ is \emph{tame on $U$}
if $(M,\lambda)$ admits a \emph{contact $J$-quasi pseudoconvex pair} on $U$ with a \emph{proper}
and $\lambda$-tame $\varphi$.
\item We call an end of $(M,\lambda)$ \emph{tame} if $\lambda$ is tame on the end.
\end{enumerate}
We say an end of contact manifold $(M,\xi)$ is tame if it admits a contact form
$\lambda$ that is tame on the end of $M$.
\end{defn}

The upshot of introducing this kind of barrier functions on contact manifold is the
amenability of the maximum principle to the pair $(\psi,J)$ in the study of contact instantons.
Unlike in \cite{oh:entanglement1}, we are considering the boundary value problem of contact instantons
under the possibly noncompact Legendrian boundary condition. Because of this
we need to consider the class of \emph{tame Legendrian submanifolds}
$\Lambda$ on a tame contact manifold $(M,\lambda)$ in general which will be amenable to
the strong maximum principle.

\begin{defn}[Tame Legendrian submanifolds] Let $(M,\xi)$ be a tame
contact manifold equipped with a tame contact form  $\lambda$
and a $\lambda$-tame function $\psi$. We say a CR-almost complex structure
$\lambda$ is \emph{tame} to $(\psi,J)$ if the triple $(\Lambda, \psi, J)$ satisfies
\be\label{eq:bdy-tame-triple}
d\psi\circ J(T\Lambda) = 0.
\ee
\end{defn}
Another way of phrasing this definition is that for such a triple 
$\grad_g \psi$ is tangent to $T\Lambda$ with respect to the triad metric
$g = g_{(\lambda, J)}$.

This notion is the contact counterpart of the notion of
\emph{gradient-sectorial} Lagrangian submanifolds in Liouville sectors
introduced by the present author in \cite{oh:gradient-sectorial}.

The following is the counterpart of \cite[Theorem 1.11]{oh:entanglement1} for the case of
bordered contact instantons \emph{under the Legendrian boundary condition}.

\begin{thm} Let $(M,\xi)$ be a contact manifold and consider the contact triad
$(M,\lambda,J)$ associated to it.
Let $\varphi$ be a $\lambda$-tame contact $J$-convex
function. Assume that $\Lambda$ is a Legendrian tangle such that the triple $(\Lambda,\varphi,J)$ is also tame.
Then for any contact instanton $w: \dot \Sigma \to M$ for the
triad $(M,\lambda, J)$,
\begin{enumerate}
\item  the composition $\varphi\circ w$ is
a subharmonic function, i.e., satisfies
$$
\Delta(\varphi\circ w) \geq 0.
$$
\item $\varphi\circ u$ satisfies strong maximum principle along the boundary $\del \dot \Sigma$.
\end{enumerate}
\end{thm}
\begin{proof} Statement (1) is proved in the proof of \cite[Theorem 1.11]{oh:entanglement1}.
Therefore it remains to show that $\varphi \circ u$ also satisfies strong maximum principle.

Suppose to the contrary that $\varphi \circ u$ achieves a (local) maximum at
a point $z_0 \in \del \dot \Sigma$. Choose any isothermal coordinate $x + i y$ centered at $z_0$
such that $\frac{\del}{\del x}$ is tangent to $\del \dot \Sigma$ and
$$
\frac{\del}{\del y}\Big|_{z_0} = \frac{\del}{\del \nu}.
$$
Then we have
$$
j\frac{\del}{\del \nu} = \frac{\del}{\del x}
$$
at $z_0$. Then we have
\bea\label{eq:normal-derivative}
\frac{\del}{\del \nu}(\varphi \circ u)(z_0) & = & d\varphi du \left(\frac{\del}{\del \nu}\right) =
d\varphi du \left(j \frac{\del}{\del x}\right) \nonumber \\
& = & d\varphi\left(J \frac{\del u}{\del x}(z_0)\right) = 0
\eea
where the last equality follows from the following 3 facts combined:
\begin{itemize}
\item $\frac{\del u}{\del x}(z_0)$ is tangent to $\Lambda$ at $z_0$.
\item $T_{u(z_0)}\Lambda \subset \xi_{u(z_0)}$, and $J \xi \subset \xi$.
\item The triple $(\Lambda, \varphi, J)$ is a tame triple.
\end{itemize}
Then \eqref{eq:normal-derivative} violates the strong maximum principle, unless the map
$u$ is constant in which case there is nothing to prove.
This finishes the proof.
\end{proof}

\begin{exm} Each fiber $J^1_qN$ of the 1-jet bundle $J^1N$ with respect to 
the standard contact form $dz - \theta$
is tame with respect to any CR pseudo-convex pair $(\varphi,J)$ in the sense of \cite{oh:sectorial}.
Or  more generally any exact cylindrical Lagrangian submanifold $L$ in a Liouville manifold $(X,d\alpha)$ has a canonical lift
to a tame Legendrian submanifold $R$ in the contactization $M = X \times \R$ equipped with a contact form $\lambda = dt + \alpha$. 
\end{exm}

\section{Off-shell description of moduli spaces}
\label{subsec:off-shell}

For the exposition of this section,  we first consider general bordered
compact surfaces of arbitrary genus to be consistent with that of \cite{oh:contacton} (for the closed case).

We will be mainly interested in the two cases:
\begin{enumerate}
\item A generic nondegenerate case of $R_1, \cdots, R_k$ which in particular
are mutually disjoint,
\item The case where $R_1, \cdots, R_k = R$.
\end{enumerate}
The  second case is transversal in the Bott-Morse sense both for the Reeb
chords and for the moduli space of contact instantons, which is
rather straightforward and easier to handle, and so omitted.

We choose a $\lambda$-adapted CR-almost complex structure $J$.
Let $(\Sigma, j)$ be a bordered compact Riemann surface, and let $\dot \Sigma$ be the
punctured Riemann surface with $\{z_1,\ldots, z_k \} \subset \del \Sigma$.
For each given such on $(\dot \Sigma,j)$,
we regard the assignment
$$
\Upsilon: u \mapsto \left(\delbar^\pi u, d(u^*\lambda \circ j)\right), \quad
\Upsilon: = (\Upsilon^\pi,\Upsilon^\perp)
$$
as a section of a (infinite dimensional) vector bundle, and define the moduli space
$$
\widetilde \CM((\dot \Sigma,\del \dot \Sigma),(M,\vec R);J):= \Upsilon^{-1}(0), 
\quad \vec R = (R_1,\cdots, R_k)
$$
of finite energy maps $w: \dot \Sigma \to M$ satisfying the equation
\eqref{eq:contacton-Legendrian-bdy}. 

For the first case, all the asymptotic Reeb chords  are nonconstant and have nonzero
action $T \neq 0$.
 In particular, the relevant punctures $z_i$
are not removable. Therefore we have the decomposition of the finite energy moduli space
$$
\CM((\dot \Sigma,\del \dot \Sigma),(M,\vec R);J) =
\bigcup_{\vec \gamma \in \prod_{i=0}^{k-1}\frak{X}(R_i,R_{i+1})}
\CM((\dot \Sigma,\del \dot \Sigma),(M,\vec R);J;\vec \gamma)
$$
by the asymptotic convergence result from \cite{oh:contacton-Legendrian-bdy}.
Depending on the choice of strip-like coordinates we divide the punctures
$$
\{z_1, \cdots, z_k\} \subset \del \Sigma
$$
into two subclasses
$$
p_1, \cdots, p_{s^+}, q_1, \cdots, q_{s^-} \in \del \Sigma
$$
as the positive and negative boundary punctures. We write $k = s^+ + s^-$.

Let $\gamma^+_i$ for $i =1, \cdots, s^+$ and $\gamma^-_j$ for $j = 1, \cdots, s^-$
be two given collections of Reeb chords at positive and negative punctures
respectively. Following the notations from \cite{behwz}, \cite{bourgeois}
(but applied to the Reeb chords instead of closed Reeb orbits),
we denote by $\underline \gamma$ and $\overline \gamma$ the corresponding
collections
\beastar
\underline \gamma & = & \{\gamma_1^+,\cdots, \gamma_{s^+}^+\} \\
\overline \gamma & = & \{\gamma_1^-,\cdots, \gamma_{s^-}^-\}.
\eeastar
For each $p_i$ (resp. $q_j$), we associate the strip-like
coordinates $(\tau,t) \in [0,\infty) \times [0,1]$ (resp. $(\tau,t) \in (-\infty,0] \times [0,1]$)
on the punctured disc $D_{e^{-2 \pi K_0}}(p_i) \setminus \{p_i\}$
(resp. on $D_{e^{-2 \pi K_0}}(q_i) \setminus \{q_i\}$) for some sufficiently large $K_0 > 0$.

\begin{defn}\label{defn:Banach-manifold} We define
\be\label{eq:offshell-space}
\CF((\dot \Sigma, \del \dot \Sigma),(M, \vec R);\underline \gamma,\overline \gamma)
=: \CF(M,\vec R;\underline \gamma, \overline \gamma)
\ee
to be the set of smooth maps satisfying the boundary condition
\be\label{eq:bdy-condition}
w(z) \in R_i \quad \text{ for } \, z \in \overline{z_{i-1}z_i} \subset \del \dot \Sigma
\ee
and the asymptotic condition
\be\label{eq:limatinfty}
\lim_{\tau \to \infty}w((\tau,t)_i) = \gamma^+_i(T_i(t+t_i)), \qquad
\lim_{\tau \to - \infty}w((\tau,t)_j) = \gamma_j^-(T_j(t-t_j))
\ee
for some $t_i, \, t_j \in [0,1]$, where
$$
T_i = \int_0^1 (\gamma^+_i)^*\lambda, \quad T_j = \int_0^1 ( \gamma^-_j)^*\lambda.
$$
Here $t_i,\, t_j$ depends on the given analytic coordinate and the parameterizations of
the Reeb chords.
\end{defn}
We will fix the domain complex structure $j$ and its associated K\"ahler metric $h$.
We regard the assignment
\be\label{eq:Upsilon}
\Upsilon: w \mapsto \left(\delbar^\pi w, d(w^*\lambda \circ j)\right), \quad
\Upsilon: = (\Upsilon_1,\Upsilon_2)
\ee
as a section of the (infinite dimensional) vector bundle:
We first formally linearize and define a linear map
\be\label{eq:DUpsilonw}
D\Upsilon(w): \Omega^0(w^*TM,(\del w)^*T\vec R) \to \Omega^{(0,1)}(w^*\xi) \oplus \Omega^2(\Sigma)
\ee
where we have the tangent space
$$
T_w \CF = \Omega^0(w^*TM,(\del w)^*T\vec R).
$$

Then the following explicit formula thereof is derived in \cite{oh:contacton} for the closed
string case which also holds with Legendrian boundary condition added.
(We alert readers that the notation $B$ used here is nothing to do with the $B$ used before
to denote the second fundamental form.)

\begin{thm}[Theorem 10.1 \cite{oh:contacton}; See also Theorem 1.15
\cite{oh-savelyev}] \label{thm:linearization} In terms of the decomposition
$du = d^\pi u + u^*\lambda\otimes R_\lambda$
and $Y = Y^\pi + \lambda(Y) R_\lambda$, we have
\bea
D\Upsilon^\pi(u)(Y) & = & \delbar^{\nabla^\pi}Y^\pi + B^{(0,1)}(Y^\pi) +  T^{\pi,(0,1)}_{du}(Y^\pi)\nonumber\\
&{}& \quad + \frac{1}{2}\lambda(Y) (\CE_{R_\lambda}J)J(\del^\pi u)
\label{eq:DUpsilon1}\\
D\Upsilon^\perp(u)(Y) & = &  - \Delta (\lambda(Y))\, dA + d((Y^\pi \rfloor d\lambda) \circ j)
\label{eq:DUpsilon2}
\eea
where $B^{(0,1)}$ and $T_{du}^{\pi,(0,1)}$ are the $(0,1)$-components of $B$ and
$T_{du}^\pi$, where $B, \, T_{du}^\pi: \Omega^0(u^*TM) \to \Omega^1(u^*\xi)$ are
 zero-order differential operators given by
\be\label{eq:B}
B(Y) =
- \frac{1}{2}  u^*\lambda \otimes \left((\CE_{R_\lambda}J)J Y\right)
\ee
and
\be\label{eq:torsion-du}
T_{du}^\pi(Y) = \pi T(Y,du)
\ee
respectively.
\end{thm}

For the simplicity of notation, we also introduce the following notation.
\begin{notation}[Codomain of $\Upsilon$]\label{nota:CD} For given fixed $\vec R$, we define
\be\label{eq:CDw}
\CC\CD_{(J,\vec R),(j,w)}: = \Omega_J^{(0,1)}(w^*\xi) \oplus \Omega^2(\Sigma)
= \CH^{\pi(0,1)}_w \oplus \Omega^2(\Sigma)
\ee
and
\beastar
\CC\CD_{(J,\vec R)} & = & \bigcup_{(j,w) \in \CF} \{(j,w)\} \times \CC\CD_{(J,\vec R),(j,w)}\\
\CC\CD_{\vec R} & = & \bigcup_{J,(j,w) \in \CF} \{J\} \times \CC\CD_{(J,\vec R)}.
\eeastar
Here $\CC\CD$ stands for `codomain'.
\end{notation}

Since we will not vary $\vec R$ in the present discussion, we omit $\vec R$
from notation and simply write $\CC\CD = \CC\CD_{\vec R}$.
We also denote by $\CC\CD^{\text{\rm univ}}$ as the further union
\be\label{eq:CD-univ}
\CC\CD^{\text{\rm univ}} = \bigcup_{(J,\vec R)}\{(J,\vec R)\} \times  \CC\CD_{(J,\vec R)}.
\ee
Let $k \geq 2$ and $p > 2$. We denote by
\be\label{eq:CWkp}
\CW^{k,p}: = \CW^{k,p}((\dot \Sigma, \del \dot \Sigma),(M, \vec R); \underline \gamma,\overline \gamma)
\ee
the completion of the off-shell function space \eqref{eq:offshell-space}.
It has the structure of a Banach manifold modeled by the Banach space given by the following

\begin{defn}[Tangent space $T_w\CW^{k,p}$]\label{defn:tangent-space} We define
$$
W^{k,p} (w^*TM, (\del w)^*T\vec R; \underline \gamma,\overline \gamma)
$$
to be the set of vector fields $Y = Y^\pi + \lambda(Y) R_\lambda$ along $w$ that satisfy
\be\label{eq:tangent-element-pi}
\begin{cases}
Y^\pi \in W^{k,p}\left((\dot\Sigma, \del \dot \Sigma), (w^*\xi, (\del w)^*T\vec R)\right), \\
Y^\pi(z)\in T\vec R\quad \text{for }\, z \in \del \dot \Sigma
\end{cases}
\ee
and
\be\label{eq:tangent-element-lambda}
\begin{cases}
\lambda(Y) \in W^{k,p}((\dot \Sigma, \del \dot \Sigma),(\R, \{0\})),\\
\lambda(Y)(z) = 0 \quad \text{for }\, z \in \del \dot \Sigma
\end{cases}
\ee
\end{defn}
Here we use the splitting
$$
TM = \xi \oplus \span_\R\{R_\lambda\}
$$
where $\span_\R\{R_\lambda\}: = \CE$ is a trivial line bundle and so
$$
\Gamma(w^*\CE) \cong C^\infty\left((\dot \Sigma, \del \dot \Sigma), (\R,\{0\})\right).
$$
The above Banach space is decomposed into the direct sum
\be\label{eq:tangentspace}
W^{k,p}((\dot\Sigma,\del \dot \Sigma),( w^*\xi, (\del w)^*T\vec R))
\bigoplus W^{k,p}((\dot \Sigma,\del \dot \Sigma), ( \R, \{0\})) \otimes R_\lambda :
\ee
by writing $Y = (Y^\pi, g R_\lambda)$ with a real-valued function $g = \lambda(Y(w))$ on $\dot \Sigma$.
Here we measure the various norms in terms of the triad metric of the triad $(M,\lambda,J)$.

Now for each given $J$ and $w \in \CW^{k,p}((\dot \Sigma,\del \dot \Sigma), (M, \vec R);\underline \gamma,\overline \gamma)$,
we consider the Banach space
$$
\Omega^{(0,1)}_{k-1,p}(w^*\xi;J): = W^{k-1,p}(\Lambda_J^{(0,1)}(w^*\xi))
$$
the $W^{k-1,p} $-completion of $\Omega^{(0,1)}(w^*\xi) = \Gamma(\Lambda^{(0,1)}(w^*\xi))$ and form the bundle
\be\label{eq:CH01}
\CH_{k-1,p}^{(0,1)}(M,J): = \bigcup_{w \in \CW^{k,p}} \Omega^{(0,1)}_{k-1,p}(w^*\xi;J)
\ee
over $\CW^{k,p}$.

\begin{defn}\label{defn:CHCM01} We associate the Banach space
\be\label{eq:CH01-w}
\CC\CD^{(0,1)}_{k-1,p}(M,\lambda;J)|_w: = \Omega^{(0,1)}_{k-1,p}(w^*\xi;J) \oplus \Omega^2_{k-2,p}(\dot \Sigma)
\ee
to each $w \in \CW^{k,p}$ and form the bundle
\beastar
\CC\CD^{(0,1)}_{k-1,p}(M,\lambda;J)& : = & \bigcup_{w \in \CW^{k,p}} \CC\CD^{(0,1)}_{k-1,p}(M,\lambda)|_w\\
& \cong & \CH_{k-1,p}^{(0,1)}(M,J) \bigoplus \left(\CW^{k,p} \times \Omega^2_{k-2,p}(\dot \Sigma)\right)
\eeastar
over the Banach manifold $\CW^{k,p}$ given in \eqref{eq:CWkp}.
\end{defn}

Then we can regard the assignment
$$
\Upsilon_1: w \mapsto \delbar^\pi w
$$
as a smooth section of the bundle $\CH_{k-1,p}^{(0,1)}(M,\lambda) \to \CW^{k,p}$. Furthermore
the assignment
$$
\Upsilon_2: w \mapsto d(w^*\lambda \circ j)
$$
defines a smooth section of the trivial bundle
$$
\Omega^2_{k-2,p}(\Sigma) \times \CW^{k,p} \to \CW^{k,p}.
$$
We summarize the above discussion into the following lemma.

\begin{lem}\label{lem:Upsilon} Consider the vector bundle
$$
\CC\CD_{k-1,p}(M, \vec R;J) \to \CW^{k,p}.
$$
The map $\Upsilon$ continuously extends to a continuous section still denoted by
$$
\Upsilon: \CW^{k,p} \to \CC\CD_{k-1,p} (M,\vec R;J).
$$
\end{lem}

With these preparations, the following is a consequence of the exponential estimates established
in \cite{oh:contacton-Legendrian-bdy}.
(See \cite{oh-wang:CR-map1} for the closed string case of vanishing charge.)

\begin{prop}\label{prop:exp-decay}
Assume $\lambda$ is nondegenerate.
Let $w:\dot \Sigma \to M$ be a contact instanton and let $w^*\lambda = a_1\, d\tau + a_2\, dt$.
Suppose
\bea
\lim_{\tau \to \infty} a_{1,i} = 0, &{}& \, \lim_{\tau \to \infty} a_{2,i} = T(p_i)\nonumber\\
\lim_{\tau \to -\infty} a_{1,j} = 0, &{}& \, \lim_{\tau \to -\infty} a_{2,j} = T(q_j)
\eea
at each puncture $p_i$ and $q_j$.
Then $w \in \CW^{k,p}(M, \vec R;\underline \gamma,\overline \gamma)$.
\end{prop}

Now we are ready to define the moduli space of contact instantons with prescribed
asymptotic condition.
\begin{defn}\label{defn:tilde-modulispace} Consider the zero set of the section $\Upsilon$
\be\label{eq:defn-tildeMM}
\widetilde \CM(M,\lambda, \vec R;\underline \gamma,\overline \gamma;J) =  \Upsilon^{-1}(0)
\ee
in the Banach manifold $\CW^{k,p}(M,\lambda,\vec R;\underline \gamma,\overline \gamma)$.
We write $w \sim w'$ for two elements therefrom if there is a biholomorphism
$\varphi$ of the punctured bordered Riemann surfaces $(\dot \Sigma, \del \dot \Sigma)$ such that
$w' = w \circ \varphi$, and define the quotient space and
\be\label{eq:defn-MM}
\CM(M,\lambda, \vec R;\underline \gamma,\overline \gamma;J)
= \widetilde \CM(M,\lambda, \vec R;\underline \gamma,\overline \gamma;J)/\sim
\ee
to be the set of equivalence classes of contact instantons $w$ under the equivalence relation $\sim$.
\end{defn}
This definition does not depend on the choice of $k$ or $p$  as long as $k\geq 2, \, p>2$.
We call an equivalence class $[w]$ an isomorphism class of contact instantons and often just write it as $w$
by an abuse of notation whose meaning should be clear from the give context.

Recall that the glued domain is homeomorphic to
the original domain $\Sigma$ and so 
$$
\chi(\widetilde \Sigma) = \chi(\Sigma) = 1 - \widetilde g = 2- 2g -h
$$
where $\widetilde g$ is the genus of the complex double $\Sigma_\C$ of $\Sigma$ \cite{katz-liu}.

The following index formula is proved in \cite{oh-yso:index}.
 
\begin{thm}[Theorem 11.4 \cite{oh-yso:index}]\label{thm:index-formula} 
Let $u$ be a contact instanton 
satisfying \eqref{eq:bdy-condition} and \eqref{eq:limatinfty}. Then
we have
\beastar
\Index D\Upsilon(u) & = & 
\Index L_{\rm glued} \\
&{}& - \sum_{i,\alpha}\Index L_{w_{i,\alpha}^+} +  \sum_{j,\beta} \Index L_{w_{j,\beta}^-} \\
&=& \mu\left(E_{u;(\vec \ell,\vec w)}, \alpha_{u;(\vec \ell,\vec w)}\right) 
+ n (2-2g-h)\\
& {} & - \sum_{i,\alpha}\mu\left([w_{i,\alpha}^+,\gamma^+_{i,\alpha}];\alpha^+_{i,\alpha}\right) 
+ \sum_{j,\beta}\mu\left([w_{j,\beta}^-,\gamma^-_{j,\beta}];\alpha^-_{j,\beta}\right). \\
\eeastar
\end{thm}

Assuming $h \geq 1$ and the stability condition i.e., 
$$
g \geq 1, \text{\rm or }\, (g,h) = (0,1), \, k \geq 3, \,  \text{\rm or }\, 
(g,h) = (0,2), \, m \geq 1,
$$
the dimension of the moduli space
of bordered stable curves with $k$ marked points
$$
\CM_{(g;h),m}:= \left\{ ((\Sigma,j), (z_1, \cdots, z_m)) \mid \Sigma \, \text{of type $(g;h)$} \right\}
\Big\slash\Aut(\Sigma,j)
$$
is given by
$$
\text{\rm vit. dim} \CM(\dot \Sigma,\CE;\vec \gamma;B) = 
\Index D\Upsilon(u) + \dim \CM_{(g;h),m}.
$$
Furthermore  we have the formula
\bea\label{eq:dim-CMghk}
\dim \CM_{(g;h),m} & = & (\chi(\Sigma)+ m-1) - \dim \Aut(\Sigma, \vec z) \nonumber\\
& = & \begin{cases}  1-2g-h+m  \quad & \text{\rm unless $(g,h) = (0,1)$ or $(g,h) = (0,2)$}\\
m-1 \quad & \text{\rm if $(g,h) = (0,2)$}\\
m-3 \quad & \text{\rm if $(g,h) = (0,1)$}.
\end{cases}
\eea
Combining Theorem \ref{thm:index-formula} and \eqref{eq:dim-CMghk}, we can explicitly
compute the  virtual dimension of the moduli space.

\part{Contact instanton DGAs of Legendrian links}

By now, we have a pretty complex system of `networks' of asymptotic Reeb chords and 
bridges joining components of the Legendrian link $\vec R$ and various contact instantons
and off-shell maps $u$ and $w$ carrying their asymptotic chords and bridges. Because of its complexity,
enumerating them needs some strategy not only for the clarity of exposition but also to properly
encode their \emph{entanglement structure}. Our strategy is to enumerating them 
in the following order:
\begin{enumerate}
\item We start with enumerating the 
components of the link $\vec R$ and the given marked points of $\del \Sigma$ (or equivalently
the punctures of $\del \dot \Sigma$). These are independent of the contact instanton map $u$.
\item For given contact instanton $u$, we enumerate them according to those of the 
punctures of $\dot \del \Sigma$ respectively.
\item Then for each given asymptotic Reeb chord $\gamma$ of $u$, we use the values 
$s(\gamma)$ and $t(\gamma)$ of the source and the target maps $s$ and $t$ respectively
and enumerate the associated Legendrian submanifolds by $R_{s(\gamma)}$ and $R_{t(\gamma)}$.
\item Next we enumerate the bridge from $R_{s(\gamma)}$ to $R_{t(\gamma)}$ by
$$
\ell_{s(\gamma)t(\gamma)}
$$
which was given when we equip the link $\vec R$ with the bridges between its connected components.
\item Finally we use the enumeration of the bounding square $[w,\gamma]$ using the enumeration of
the asymptotic chord $\gamma$ and omitting its dependence on the associated bridge $\ell$ at $t=1$.
This is because the enumeration of $\ell$ is automatically given by the associated asymptotic chord $\gamma$.
\end{enumerate}

\section{Moduli spaces of disc contact instantons}
\label{sec:on-shell}

Let $(\Sigma, j)$ be a bordered compact Riemann surface  of disk-type with boundary marked points
$$
\vec z = \{z_1,\ldots, z_m \} \subset \del \Sigma,
$$
 and let $\dot \Sigma$ be the punctured Riemann surface. 
 
We partition $\vec z$ into $\vec z = \vec z^+ \sqcup \vec z^-$ 
with $m = s^+ + s^-$ as in \eqref{eq:vec-zj}. 
We specialize to the case
\be\label{eq:null}
s^- = 1, \, s^+ = k, \quad m = k+1
\ee
which enters in the construction of contact instanton cohomology \cite{oh-yso:spectral}, and
will enter in the construction of a Fukaya-type category of contact manifolds in the present paper. 
We set $z_0$ to be a positive puncture and the rest of $z_i$'s for $i = 1, \cdots k$ to be negative ones
adopting the enumeration strategy of the book \cite{fooo:book1,fooo:book2}.

Let $\vec{R} = (R_1,\ldots,R_k)$ be a Legendrian link and 
define the set of maps
$$
C^\infty(\dot \Sigma,\vec R;\vec \gamma) 
= C^\infty(\dot \Sigma,\vec R;\underline \gamma,\overline \gamma), \quad \vec \gamma = \underline \gamma \cup \overline \gamma
$$
to be the set of all smooth $u: \dot \Sigma \to M$ on $\dot \Sigma$ such that 
\be\label{eq:uzjpj}
\begin{cases}
 u\left((\infty,t)_{z^+_{i,\alpha}}\right) = \gamma^+_{i,\alpha}(t)  \\
 u(e^i_{\alpha-1}) \subset R_{s(\gamma^+_{i,\alpha})}, \quad u(e^i_{\alpha}) \subset R_{t(\gamma^+_{i,\alpha})} 
 \end{cases}
 \ee
 where we have
 $$
 \gamma^+_{i,\alpha} \in 
 \mathfrak{X}\left(R_{s(\gamma^+_{i,\alpha})},R_{t(\gamma^+_{i,\alpha})}\right)
 $$
 at the positive puncture $p^+_{i,\alpha}$ and that similar conditions hold at the negative
 punctures $p^-_{j,\beta}$. 
 
We equip $\vec R$ with a system of graded bridges 
$$
\vec \ell = \{\ell_{ab}\}_{a,b \in \underline k}.
$$
For a given collection of Reeb chords $\vec \gamma = (\underline \gamma,\overline \gamma)$,
we write
\be\label{eq:ell+ialpha}
\ell^+_{i,\alpha}: = \ell_{s(\gamma^+_{i,\alpha})t(\gamma^+_{i,\alpha})}
\ee
to be the bridge from $R_{s(\gamma^+_{i,\alpha})}$ to $R_{t(\gamma^+_{i,\alpha})}$ and similarly for
$\ell^-_{j,\beta}$.

Let $\mathcal{E} = (\vec R, \vec \ell)$, 
be a graded bridged Legendrian link.  We denote by
$$
\mathcal{F}((\dot{\Sigma},\del \dot{\Sigma}), \mathcal{E}; \gamma_0, \overline \gamma)
\subset C^\infty(\dot \Sigma,\vec R;\gamma_0, \overline \gamma) 
$$
the set of such maps satisfying \eqref{eq:bdy-condition}-\eqref{eq:limatinfty} 
and denote by
$$
 \pi_2(\mathcal{E};\gamma_0, \overline \gamma) 
 = \pi_2(\dot \Sigma;\vec{R},\vec \ell;\gamma_0, \overline \gamma)
 $$
the set of homotopy classes of such maps.  We denote by 
 $$
\CF(\dot \Sigma,\CE;\vec \gamma; B) \subset
C^\infty(\dot \Sigma,\CE;\gamma_0, \overline \gamma)
$$ 
the subset consisting of the maps $u$ in class 
$ [u] = B \in \pi_2(\dot \Sigma,\CE;c, \overline b) = :\pi_2(\CE;\gamma_0, \overline \gamma)$.
We recall the nonlinear section $\Upsilon$ given in \eqref{eq:Upsilon}
and consider its extension to the Sobolev space
$$
\CW^{k,p} := \CW^{k,p}((\dot{\Sigma},\del \dot{\Sigma}), \mathcal{E}; \gamma_0, \overline \gamma)
$$
restricted to the subset of elements $u$ in class $[u]=B$.

We will fix a domain complex structure $j$ and its associated K\"ahler metric $h$.
Before launching on our computation,  we first
remind readers of the decomposition $Y = Y^\pi + \lambda(Y)\, R_\lambda$.
Noting that $Y^\pi$ and $\lambda(Y)$ are independent of each other, we will write
\be\label{eq:Y-decompose}
Y = Y^\pi + f R_\lambda, \quad f: = \lambda(Y)
\ee
where $f: \dot \Sigma \to \R$ is an arbitrary function satisfying the boundary condition
\be\label{eq:Y-bdy-condition}
Y^\pi(\del \dot \Sigma) \subset T\vec R, \quad f|_{\del \dot \Sigma} = 0
\ee
by the Legendrian boundary condition satisfied by $Y$.

By this decomposition and the associated decomposition of the codomain,
we are led to the linearized operator of the form
$$
D\Upsilon(u): \Omega^0(u^*TM,(\del u)^*T\vec R) \to \Omega^{(0,1)}(u^*\xi) \oplus \Omega^2(\Sigma)
$$
of each contact instanton $u$, i.e., a map $u$ satisfying $\Upsilon(u) = 0$.
We regard the assignment
\be\label{eq:Upsilon}
\Upsilon: u \mapsto \left(\delbar^\pi u, d(u^*\lambda \circ j)\right), \quad
\Upsilon: = (\Upsilon^\pi,\Upsilon^\perp)
\ee
as a section of the (infinite dimensional) vector bundle: Using the decomposition
$TM = \xi \oplus \R\langle R_\lambda \rangle$, we write $\Upsilon = (\Upsilon^\pi, \Upsilon^\perp)$.
Then we can regard the assignment
$$
\Upsilon^\pi: u \mapsto \delbar^\pi u
$$
as a smooth section of some infinite dimensional vector bundle and
the assignment
$$
\Upsilon^\perp: u \mapsto d(u^*\lambda \circ j)
$$
as a map to $\Omega^2(\Sigma)$.
(We refer to \cite{oh:contacton-transversality} for the precise off-shell framework
for these operators.) 
Then we consider the moduli space
$$
\CM((\dot{\Sigma},\del \dot{\Sigma}), \mathcal{E}; \gamma_0, \overline \gamma)
 : = \Upsilon^{-1}(0), \quad \vec R = (R_0,\cdots, R_k)
$$
 of finite energy maps $u: \dot \Sigma \to M$ satisfying the equation
\eqref{eq:contacton-Legendrian-bdy}.

Then we recall the following lemmata from \cite{fooo:anchored}.

\begin{lem}[Compare with Lemma 5.8 \cite{fooo:anchored}]\label{lem:dim-deg} Let 
$\gamma, \, w, \, \alpha_{01}$
be as above, and put
$$
w^+(s,t) = w(s,1-t), \quad \lambda_{10}(t) = \alpha_{01}(1-t).
$$
Then
\be\label{eq:degPD}
\mu([\gamma,w];\alpha_{01}) + \mu([\widetilde \gamma,w^+];\alpha_{10}) = n.
\ee
\end{lem}

\begin{lem}[Compare with Lemma 5.12 \cite{fooo:anchored}]\label{lem:poly} Let $\CE$ be a graded bridged Legendrian link and
assume that the graded base path $(\ell_{0k},\lambda_{0k})$ is given by \eqref{eq:pi2-polygonal}
for given $(\ell_{i(i+1)},\gamma_{i(i+1)})$ for $i = 0, \ldots, k-1$ and $B \in \pi_2({\CE},\vec \gamma)$. Then
$$
\mu(\CE,\vec \gamma;B) + \sum_{i=0}^{k-1}
 \mu\left([\gamma_{i(i+1)},w_{i(i+1)}];\alpha_{i(i+1)}\right) =
\mu([\gamma_{0k},w_{0k}];\alpha_{0k}).
$$
\end{lem}

The  condition \eqref{eq:null} implies 
$$
\mu\left(E_{u;[w,\vec \ell]}, \alpha_{u;[w,\vec \ell]}\right) = 0.
$$
The following special case of the index formula immediately follows combining the
above discussions.

\begin{cor}[Corollary 11.5 \cite{oh-yso:spectral}]\label{cor:disc-index}
 Suppose $\Sigma = D^2$ and the condition \eqref{eq:null} with $k+1$ 
boundary marked points.  Then we have
\bea
\Index D\Upsilon(u) & = & n + \mu\left([w_{k(k+1)}^-,\gamma_{k+1}];\alpha^-_{k(k+1)}\right)
 - \sum_{i=1}^{k}\mu\left([w_{(i-1)i}^+,\gamma_i];\alpha^+_{(i-1)i}\right)
\nonumber\\
& = &  \mu \left([w_{0k}^+,\widetilde \gamma_0];\alpha^+_{0k}\right)
 - \sum_{i=0}^{k}\mu \left([w_{(i-1)i}^+,\gamma_i];\alpha^+_{(i-1)i}\right)
\eea
where we take the identification $k+1 \equiv 0 \mod k+1$, and $\widetilde \gamma_j$
and $w^+_{0k}(\tau,t)$ are the 
maps defined by 
$$
\widetilde \gamma_0(t): = \gamma_0(1-t), \quad w^+_{0k}(\tau,t) := w_{k(k+1)}^-(-\tau,1-t).
$$
\end{cor}

Using these results, we prove the following relationship between the shifted degrees.

\begin{prop}[Lemma 5.2.24 \cite{fooo:anchored}]\label{dim-deg}
Whenever $\dim \CM(\CE, \vec \gamma;B) = 0$, we have
\be\label{misdeg1}
(\mu([\gamma_{0k},w_{0k}];\alpha_{0k}) - 1)
= 1 + \sum_{i=1}^k (\mu([\gamma_{i(i+1)},w_{i(i+1)}];\alpha_{i(i+1)}) - 1).
\ee
\end{prop}
\begin{proof}
Lemma \ref{lem:poly} and Corollary \ref{cor:disc-index} imply
$$
\sum_{i=0}^{k}
\mu([\gamma_{i(i+1)},w_{i(i+1)}];\alpha_{i(i+1)})
= n + k -2
$$
in the case $\dim \CM(\CE, \vec \gamma;B) = 0$.
By Lemma \ref{lem:dim-deg}, we have
$$
\mu([\gamma_{0k},w_{0k}];\alpha_{0k})=
- \mu([\gamma_{k0},w_{k0}];\alpha_{k0}) + n.
$$
Substituting this into the above identity and rearranging the identity, we obtain
the proposition.
\end{proof}

\section{Definition of chain modules}

Let $(R,R')$ be a pair of connected Legendrian submanifolds with graded bridge $(\ell;\alpha)$. 
Assume the nondegeneracy spelled out in Hypothesis \ref{hypo:nondegeneracy}. 
For each such a pair, we will associate a $\K[q,q^{-1}]$ module over $\mathfrak{X}(R,R')$ denoted by
$$
CI_\lambda^*(R,R')
$$
by considering two cases, $R \cap R' = \emptyset$ and $R = R'$ (or more generally the case of 
clean intersections $R \cap R'$) separately.
We recall that for a generic choice of pairs $(R,R')$, we have
$$
R \cap R' = \emptyset
$$
by the dimensional reason.
(Here $CI_\lambda$ stands for \emph{contact instanton for $\lambda$} as well as the letter $C$ also
stands for \emph{complex} at the same time.)

\subsection{The case $R \cap R' = \emptyset$}

In this case, we have one-to-one correspondence between $\mathfrak{X}(R,R';\lambda)$
and the set of Reeb chords joining $R, \, R'$: the correspondence is given by
$$
(\gamma,T) \Leftrightarrow (\gamma_T, \text{\rm sign} T)
$$
where $\gamma_T: [0,|T|] \to \R$ is the Reeb chord given by $\gamma_T(t): = \gamma(t/T)$.

\begin{defn} Let $(R,R';\ell)$ be a bridged 2-components link. Let $(\gamma,T) \in \mathfrak{X}(M;R,R')$
and $w:[0,1]^2 \to M$ be a bounding square satisfying
$$
w(0,t) = \ell(t),\, w(1,t) = \gamma(t), \, w(s,0) \in R, \, w(s,1) \in R'.
$$
We say $(w,\gamma) \sim_\ell (w',\gamma')$ if 
\begin{enumerate}
\item $\gamma = \gamma'$,
\item $[\gamma,w;\alpha] = [\gamma',w';\alpha]$ in $\pi_2(R,R';\ell;\gamma)$.
\end{enumerate}
We denote by $[w,\gamma] = [w,\gamma]_\alpha$ the equivalence class of $\sim_\ell$.
\end{defn} 

For any two pairs $(w,\gamma)$ and $(w',\gamma')$ bounding the same 
bridge $\ell$ and a Reeb chord $\gamma$, we can define the concatenation
$\overline w * w'$ defined by
$$
w * \overline w'(s,t) = \begin{cases} w(2s,t) \quad & 0 \leq s \leq \frac12\\
w'(2(\frac12 - s),t) \quad & 1/2 \leq s \leq 1
\end{cases}
$$
defines an annulus map $v: S^1 \times [0,1] \to M$ as in Remark \ref{rem:GammaRR'} and hence
an element of $\pi_1(\Omega(R,R'))$, which we denote by $A$. Then we  write
$$
[w,\gamma] = [w',\gamma'] \cdot A.
$$
It is easy to see that this defines a right action of $\pi_1(\Omega(R,R');\ell)$.
A priori this group may not be abelian. Recall the natural Maslov homomorphism
$$
\mu: \pi_1(\Omega(R,R');\ell) \to \Z; \quad A \mapsto \mu(A).
$$
\begin{defn}[Maslov period group]
We define the subgroup 
$$
\Sigma_{RR'}: = \Image \mu \subset \Z
$$
call the \emph{Maslov period group} of the pair $(R,R')$.
\end{defn}
It is easy to see that the subgroup does not depend on the choice of the base path $\ell$.

Now we equip a grading to the above bridge and consider a graded bridged 2-component link $(R,R')$
and consider the set $\mathfrak{X}(R,R')$ of iso-speed Reeb chords $(\gamma,T)$.

\begin{defn} Let $(R,R';\alpha)$ be a graded bridged 2-components link. Let $(\gamma,T) \in \mathfrak{X}(M;R,R')$
and $w:[0,1]^2 \to M$ be a bounding square satisfying
$$
w(0,t) = \ell(t),\, w(1,t) = \gamma(t), \, w(s,0) \in R, \, w(s,1) \in R'.
$$
We say $(w,\gamma) \sim_\alpha (w',\gamma')$ if 
\begin{enumerate}
\item $\gamma = \gamma'$,
\item $\mu(\gamma,w;\alpha) = \mu(\gamma',w';\alpha)$.
\end{enumerate}
We denote by $[w,\gamma] = [w,\gamma]_\alpha$ the equivalence class of $\sim_\alpha$.
We define
$$
\Sigma_{RR'}^\alpha: = \{k \in \Z \mid k = \mu([w,\gamma];\alpha), (w,\gamma) \in \widetilde{\mathfrak{X}}(R,R';\alpha)\}
$$
call it the \emph{Maslov period set} of the graded bridged link $(R,R';\alpha)$.
\end{defn} 
Unlike $\Sigma_{RR'}$, $\Sigma_{RR'}^\alpha$ depends on the choice of $\alpha$.

For any two pairs $(w,\gamma)$ and $(w',\gamma')$ bounding the same graded
bridge $\ell$ and a Reeb chord $\gamma$, we can define the concatenation
$\overline w * w'$ defined by
$$
w * \overline w'(s,t) = \begin{cases} w(2s,t) \quad & 0 \leq s \leq \frac12\\
w'(2(\frac12 - s),t) \quad & 1/2 \leq s \leq 1
\end{cases}
$$
defines an annulus map $v: S^1 \times [0,1] \to M$ as in Remark \ref{rem:GammaRR'} and hence
an element of $\pi_1(\Omega(R,R'))$, which we denote by $A$. Then we write
$$
[w,\gamma] = [w',\gamma'] \cdot A.
$$
It is easy to see that this defines a right action of 
$$
\Gamma_{RR'} \cong \frac{\pi_1(\Omega(R,R'))}{\ker (\mu: A \mapsto \mu(A))}
$$
on  $\widetilde{\mathfrak{X}}(R,R';\alpha)$. 

Therefore we have the commutative diagram 
$$
\xymatrix{\widetilde{\mathfrak{X}}(R,R';\alpha) \ar[r]^{\hookrightarrow} \ar[d]_\pi 
& \widetilde \Omega(R,R')\ar[d]_\pi\\
{\mathfrak{X}}(R,R';\alpha) \ar[r]^{\hookrightarrow} & \Omega(R,R')}
$$
whose deck transformation group is given by $\Gamma_{RR'}$. We consider the
group ring
$$
\K[\Gamma_{RR'}]: = \left\{ \sum_{i=1} k_i q^{\mu(A_i)} \mid A_i \in \Gamma_{RR'}, \, k_i \in \K \right\}.
$$
By definition, this is a subring of $\K[q,q^{-1}]$.

\begin{rem} We would like to mention that the relationship between $\K[\Gamma_{RR'}]$ and
$\K[q,q^{-1}]$ is analogous between the universal Novikov ring $\Lambda_{\text{\rm nov}}$ and
the Novikov ring $\Lambda(L,L';\ell_0)$ associated to the based Lagrangian pair $(L,L')$
in symplectic geometry with the formal parameter $T$ encoding the action dropped. The latter is 
because the action in the current contact case is always single-valued.
\end{rem}

These being said, we define a $\K[q,q^{-1}]$ module
$C(R,R';\K[q,q^{-1}])$ and construct a filtered
$A_{\infty}$ bimodule structure on it. Let
\index{Floer cochain module}
\be\label{eq:CIRR'ell}
CI(R,R';\ell) = \underset{[w,\gamma] \in \widetilde{\mathfrak{X}}(R,R';\ell) }
\bigoplus \K [w,\gamma],
\ee
$$
\widehat{CI}(R,R';\ell;\K[q,q^{-1}])
:= CI(R,R';\ell)\,\,\widehat\otimes_{\Bbb R}\,\,\Lambda_{\text{\rm nov}}.
$$
and an equivalence relation $\sim$ on
$\widehat{CI}(R,R';\ell;\K[q,q^{-1}])$ as follows:
Say
$$
e^{\mu}[w,\gamma]
\sim  e^{\mu^{\prime}}[w',\gamma']
$$
for $e^{\mu}[w,\gamma]$, $ e^{\mu^{\prime}}[w',\gamma']
\in \widehat{CI}(R,R';\ell;\K[q,q^{-1}])$
if and only if the following conditions are satisfied:
\bea\label{eq:Gamma-equivalence-module}
\gamma & = & \gamma^{\prime} \nonumber\\
2\mu + \mu ([w,\gamma]) & = &
2\mu^{\prime} + \mu ([w',\gamma']).
\eea

We define $CI(R,R';\ell;\K[q,q^{-1}])$ to be the quotient
\be\label{eq:CIR0R1}
\index{$CI(R,R';\ell;\K[q,q^{-1}])$}
{\widehat{CI}(R,R';\ell;\K[q,q^{-1}])}/{\sim}
= : CI(R,R';\ell;\K[q,q^{-1}]).
\ee
The lemma enables us to define the following filtration on
$CI(R,R';\ell;\K[q,q^{-1}])$.

\index{energy filtration!for Floer complex}
\begin{defn}[Energy filtration for contact instanton complex] \label{defn:action-filtration}
We define the {\it energy filtration} on $\widehat{CI}(R,R';\ell)$ by setting
$$
F^{c}C(R,R';\ell): = \left\{[w,\gamma] \in \widehat{CI}(R,R';\ell) \, \Big| \,
\int \gamma^*\lambda \ge c \right\}
$$
for each $c \in \R$.
\end{defn}
This filtration obviously descends to one
on $\widehat{CI}(R,R';\ell; \K[q,q^{-1}])/\sim$. 

Now, we define
$$
C(R,R';\ell ; \K[q,q^{-1}])
$$
 as follows. Namely
\be\label{eq:CRR'ell}
C(R,R' ; \ell; \K[q,q^{-1}]) \cong
\underset{\gamma \in \mathfrak{X}(R,R')}
\bigoplus
\K[q,q^{-1}] \langle \gamma \rangle
\ee
as a $\K[q,q^{-1}]$ module.

We now explain the statement about the gapped condition.

\begin{defn}[ $(\Gamma,\Gamma')$-set $\Gamma_{RR'}$]\label{defn:GG'-set}
Consider the subset
$$
\Gamma_{RR'}= \{ \beta+ \alpha + \beta'  \in
\mid
\beta^{(i)} \in \Gamma_{R}, \, \alpha \in \Gamma_{RR',0}\}
$$
where $\Gamma_{RR',0} \subset  2\Z)$ is defined by
\beastar
\Gamma_{RR',0} &: = &\{ \mu (\gamma,\gamma';B)) \mid \gamma, \,\gamma
\in \mathfrak{Reeb}(R,R'), \\
B \in \pi_2(R,R;'\gamma,\gamma'), \CM(\gamma,\gamma';B) \ne \emptyset \}.
\eeastar
\end{defn}

\subsection{Clean intersection case}

We start with the maximally clean intersection case $R' = R$.
Then we have the decomposition
$$
\mathfrak{X}(R,R) = \mathfrak{X}_{\{T=0\}}(R,R)  \sqcup \mathfrak{X}_{\{T > 0\}}(R,R) \subset \Theta(R;M)
 $$
and $\mathfrak{X}_{\{T=0\}}(R,R)$ is in one-to-one correspondence by the evaluation map
$$
(\gamma,0) \mapsto \gamma(0).
$$
The following definition is the open string analogue of \cite[Definition 1.7]{bourgeois} and \cite[Definition 1.2]{oh-wang:CR-map2}.

\begin{defn}\label{defn:morse-bott}
Let $\lambda$ be a contact form and consider $\mathfrak{X}(R,R)$. We say a
connected component of $\mathfrak{X}(R,R)$ is a Morse-Bott component
in $\Theta(R;M)$ if the following holds:
\begin{enumerate}
\item The tangent space at
every pair $(T,z) \in \mathfrak{X}(R,R)$ therein coincides with $\ker d_{(\gamma,T)}\Phi_\lambda$.
\item The locus $Q \subset M$ of Reeb chords is embedded.
\item The 2-form $d\lambda|_Q$ associated to the locus $Q$ of Reeb chords has constant rank.
\end{enumerate}
\end{defn}

For the simplicity of notation, we denote
$$
\widetilde R: = \mathfrak{X}_{T=0}(R,R)
\cong \{ \gamma:[0,1] \to M \mid \gamma(t) \equiv q \in R\}.
$$
The following general proposition is proved in \cite[Proposition 10.7]{oh:contacton-gluing}.

\begin{prop}\label{prop:Bott-clean} Let $(M,\xi)$ be a contact manifold, and
let $R$ be a compact Legendrian submanifold.
Consider Moore chords $(\gamma, T)$ of $\widetilde R$ i.e., pairs
$\gamma:[0,1] \to M$ with
\be\label{eq:Moore-path-action}
\gamma(0),\, \gamma (1) \in R, \quad \int \gamma^*\lambda = T \geq 0.
\ee
Then for any choice of contact form $\lambda$ of $(M,\xi)$,
the set $\widetilde R$ is a Morse-Bott component of $\mathfrak{X}(R,R)$ in $\Theta(R;M)$.
\end{prop}

We apply this Morse-Bott property and define the Morse-Bott chain group to be
\be\label{eq:chaingroup-0}
CI_\lambda^*(R;M): = C^*(R) \oplus \K[q,q^{-1}] \langle\mathfrak{X}_{T\neq 0}(R;M)\rangle
\ee
where we take a suitable cochain complex $C^*(R)$ of $R$, e.g., the de Rham complex of
$R$ or the Morse complex of a Morse function $f$ on $R$. For the simplicity, we may take
$$
C^*(R) = CM^*(f;R)
$$
here for a $C^2$-small Morse function $f$, in particular with
$$
\|f\|_{C^0} < T_\lambda(M;R).
$$
Here the group $CM^k(M;R)$ is the free abelian group generated by
the set $\text{Crit}(f)$ of critical points of $f$ of Morse index $k$.
(See \cite{oh:entanglement1}, \cite{oh:contacton-gluing}, \cite{oh-yso:spectral} for this choice to calculate
the contribution of `short Reeb chords' in homology and its applications.)

\section{Legendrian contact instanton curved $A_\infty$ algebra}

In this section, we will first give the construction of an $A_\infty$-algebra associated to 
each graded bridged Legendrian link
\be\label{eq:E-alpha}
\CE = (\vec R, \vec \alpha); \, \vec \alpha = (\alpha_{ab})_{a,b \in \underline k}, \quad 
\alpha_{ab} \in \Gamma\left(\ell_{ab}^*\Lag(\xi)\right)
\ee
with $\vec R = (R_0,\cdots, R_k)$. We will
dualize it into the DGA that we aim to associate to the Legendrian link $\vec R$ in
the next section.

For this purpose, we 
consider the moduli space with $\dim \CM(Y,\vec R;\underline \gamma,\overline \gamma; B) = 0$.
For the clarity of enumeration of the asymptotic chords $\gamma$'s,
we change the enumeration of $\vec \gamma = \underline \gamma \cup \overline \gamma$ into
$$
\underline \gamma = \{\gamma_{01}\}, \quad \overline \gamma = \{\gamma_{12}, \cdots, \gamma_{(k-1)k}\}.
$$

\begin{defn}[Legendrian contact instanton $A_\infty$-algebra] 
The \emph{Legendrian contact instanton $A_\infty$-algebra}
is the one whose structure maps $\{\mathfrak m_k\}_{k \geq 0}$ are defined by the formula
\beastar
&{}& \mathfrak m_k([\gamma_{01},w_{01}],[\gamma_{12},w_{12}],\cdots, [\gamma_{(k-1)k},w_{(k-1)k}]) \,
\\
& = & \sum_B \mathfrak m_{k,B}([\gamma_{01},w_{01}],[\gamma_{12},w_{12}],\cdots, [\gamma_{(k-1)k},w_{(k-1)k}])
\,  [\gamma_{0k},w_{0k}] q^{\mu(B)}\\
& = & \sum_B \sum \#\left(\CM_{k+1}(R_1, \cdots R_k;[\vec \gamma,\vec w];B)\right) \,\ [\gamma_{0k},w_{0k}] q^{\mu(B)}.
\eeastar
Here the sum is over the basis $[\gamma_{0k},w_{0k}]$ of
$CI(R_k,R_0)$, where $B \in \pi_2(\CE;\vec \gamma)$ and $
\vec \gamma = (\gamma_{0k}; \gamma_{01},\gamma_{12}, \cdots, \gamma_{(k-1)k})$ 
such that
$\dim \CM_{k+1}(R_1, \cdots R_k;\vec \gamma;B) = 0$, i.e.,
\be\label{eq:dim=0}
0 = \mu \left([w_{0k},\gamma_{01}];\alpha_{0k}\right)
 - \sum_{i=0}^{k}\mu \left([w_{(i-1)i},\gamma_{(i-1)i}];\alpha_{(i-1)i}\right).
\ee
\end{defn}
Here the operator $\mathfrak m_{k,B}$ is given by the formula
\bea\label{eq:mkB}
&{}& \mathfrak m_{k,B}([\gamma_{01},w_{01}],[\gamma_{12},w_{12}],\cdots, [\gamma_{(k-1)k},w_{(k-1)k}]) 
\nonumber\\
& = & \sum \#\left(\CM_{k+1}(R_1, \cdots R_k;\vec \gamma;B)\right) \, [\gamma_{0k},w_{0k}]
\eea
when  $[\gamma_{0k},w_{0k}] =  (\sum_{i=0}^{k-1} w_{i(i+1)}) \# B$. 
Then we define the map $\mathfrak m_k$ as the sum 
$$
\mathfrak m_k = \sum_{B \in \pi_2(\vec R^0, \cdots, \vec R^k,{\bf p})} m_{k,B} \, q^{\mu(B)}.
$$
We consider their unique coderivations $\widehat{\mathfrak m_k}$
extension thereof of the bar complex $B\CA(Y,\vec R)$ and consider the dual operators
$$
\widehat{\mathfrak m} = \sum_{k=0} \widehat{\mathfrak m}_k.
$$
The following proposition shows that $\del$ is indeed a differential.

\begin{prop}\label{prop:m2=0} The map \eqref{eq:mkB} for $k=1$ satisfies $\widehat{\mathfrak m}^2 = 0$.
\end{prop}
\begin{proof} 
The formula \eqref{eq:mkB} implies that $\mathfrak m_{k,B}$ above is a $R$-linear map
over the given coefficient field (e.g., $R = \Z, \, \R, \, \C$).
%
%

Then by the compactification of moduli spaces of bordered punctured contact 
instantons with prescribed asymptotics given in \cite{oh:entanglement1},
exactly the same proof as that of \cite{fooo:book1} for the Fukaya algebra of 
a compact Lagrangian submanifold proves that $\widehat m \circ \widehat m = 0$, i.e.,
the curved $A_\infty$-relation holds.  This finishes the proof.
\end{proof}

\section{Legendrian contact instanton DGA and its augmentations}

In this section, we define the LCI-DGA of any Legendrian link by closely mimicking
the exposition of \cite[Section 3]{ekholm:rational-SFT} in which the Legendrian
contact homology was constructed by using the moduli space of pseudoholomorphic curves on symplectization. 
In this section, by adopting the notation of \cite{ekholm:rational-SFT}, we write
$$
(\gamma_{01},\overline \gamma) = (c, \overline b),
$$
i.e., $\gamma_{01}=: c$ and $ \gamma_{(i-1)i} = b_i$ for $i = 1, \cdots, k$.

\subsection{Legendrian contact instanton DGA (LCI-DGA)}

By taking the dual version of the bar complex $BCI(Y,\CE) = TCI(Y,\CE)[1]$
$$
\del: = \widehat{\mathfrak m}^*
$$
of the above $A_\infty$ algebra, we obtain the Eliashberg-Chekanov type DGA
associated to the graded bridged Legendrian link $\CE$.

\begin{defn} The DGA of $(Y,\vec R)$ is the unital noncommutative graded algebra 
$\CA(Y,\vec R;\Lambda)$ over the ring $\K$ generated by the Reeb chords of $\vec R$. The grading of the element
$[c,w]$ is given by
$$
\mu\left([c,w];\alpha_{s(c)t(c)}\right).
$$
\end{defn}
%
%

Now the construction of DGA $\CA(Y,\CE)$ is in order.
For this purpose, we consider the moduli spaces with 
$$
\text{\rm vir}\dim \CM(Y,\vec R;c,\overline b; B) = 0.
$$

\begin{defn}[Legendrian contact instanton DGA] The \emph{Legendrian contact instanton DGA (LCI DGA)}
is the DGA $\CA(Y,\CE)$ whose differential $\del: \CA(Y,\CE) \to \CA(Y,\CE)$ is defined by the formula
\be\label{eq:differential}
\del c = \sum_{\dim \CM_{1;k}(Y,\CE;c,\overline b) = 0} \# (\CM(Y,\CE;c,\overline b)) \, \overline b,
\ee
where $c$ is a Reeb chord and $\overline b = b_1b_2\cdots b_k$ is a word of Reeb chords.
\end{defn}
Then, by taking the dual of the proof of
$\widehat m \circ \widehat m = 0$, Proposition \ref{prop:m2=0} implies $\del^2 = 0$.

\subsection{Augmentations and the linearization of DGA $\CA(Y,\CE)$}

The following is the dual version of the $\mathfrak m_0$-term.

\begin{defn}[Augmentation] An \emph{augmentation} of $\CA(Y,\vec R)$ is a chain map
$\epsilon: \CA(Y,\vec R) \to R$ where $R$ is equipped with the trivial differential.
\end{defn}

Given an augmentation $\epsilon$, we define the algebra isomorphism 
$\E_\epsilon: \CA(Y,\vec R) \to \CA(Y,\vec R)$ by setting
$$
E_\epsilon(c) = c + \epsilon(c)
$$
for each generator. We consider the word-length filtration of $\CA(Y,\vec R)$,
$$
\CA(Y,\vec R) = \CA_0(Y,\vec R) \supset \CA_1(Y,\vec R) \supset \CA_2(Y,\vec R) \supset \cdots.
$$
The differential $\del^\epsilon: = E_\epsilon \circ \del \circ E_\epsilon^{-1}$
respects this filtration, 
$$
\del^\epsilon(\CA_j(Y,\vec R)) \subset \CA_j(Y,\vec R).
$$
\begin{defn} We write
$$
Q(Y,\vec R): = \CA_1(Y,\vec R)/\CA_2(Y,\vec R) \to \CA_1(Y,\vec R)/\CA_2(Y,\vec R)
$$
and call $\del^\epsilon_1:Q(Y,\vec R) \to Q(Y,\vec R)$ the \emph{$\epsilon$-linearized differential}.
The \emph{$\epsilon$-linearized contact instanton contact homology} ($\epsilon$-linearized CI contact homology)
is the $R$-module
$$
H^*_{\LCI}(Y,\vec R;\epsilon) = \ker(\del_1^\epsilon)/\Im (\del_1^\epsilon).
$$
(See \cite{oh-yso:spectral} for the proof via the Morse theory.)
\end{defn}

The following special case has been studied by Seungook Yu and the author in \cite{oh-yso:spectral}.

\begin{exm} Consider $Y = J^1B$ for a closed manifold $B$ equipped with the standard 
contact structure given by the contact form
$$
\lambda = dz - pdq.
$$
Let $R_0 = o_{J^1B}$ be the zero section and consider the collection $\vec R = \{R_0\}$.
Due to the nature of the zero section and the standard contact form, we have
 $\mathfrak{Chord}(R_0) = \emptyset$.
Therefore $(J^1B, \vec R)$ is of Morse-Bott type and all iso-speed Reeb chords have period 0 and hence of
the type $(q,0)$ for any $q \in B = \CZ_{J^1B}$. Therefore we have
$$
\CA(J^1B,o_{J^1B}) \cong (\Omega^*(B),d)
$$
with differential. Hence we have
 $H^*_{\LCI}(J^1B,o_{J^1B}) = H^*(B)$.
\end{exm}

\subsection{Lagrangian cobordisms and DGA maps}

Let $\vec R$ be a Legendrian link with their connected components
$$
(R_1, \cdots, R_k).
$$

We start with recalling the definition of a symplectic manifold of $\dim = 2n$ with cylindrical ends: It is 
a symplectic manifold $x$ that has a compact subset $K$ such that 
$$
X - K = Y^+ \times [0,\infty) \sqcup Y^- \times (-\infty,0]
$$
where each of $(Y^\pm,\lambda^\pm)$ is a finite disjoint union of compact connected contact $(2n-1)$-manifolds 
equipped with contact form $\lambda^\pm$ respectively.
More precisely, we have a symplectic embedding  
$$
\left(Y^+ \times [0,\infty), d(e^s \lambda^\pm)\right) \hookrightarrow (X \setminus K,\omega)
$$
(respectively a symplectic embedding of $Y^- \times (-\infty,0]$ thereinto, which satisfy
the above union property. We call the images of the embeddings
the positive end and the negative end of $X$ respectively.

We can also define a Lagrangian submanifold with cylindrical ends in the similar way by requiring that
there exists another compact subset $K' \subset X$ such that
$$
L \setminus K' = \vec R^+ \times [0,\infty) \sqcup \vec R^- \times (-\infty,0]
$$
where each of $\vec R^\pm$ is either empty or a Legendrian submanifold of $(Y^\pm, \lambda^\pm)$ respectively.

\begin{defn} We say a pair $(Y^\pm, \vec R^\pm)$ of $(\mathsf{Cont},\mathsf{Leg})$
\emph{Lagrangian cobordant} if there exists a pair $(X,L)$ 
of $(\mathsf{Symp},\mathsf{Lag})$ with boundary $\del X = Y^+ \cup Y^-$ (as an oriented manifold)
such that 
\begin{enumerate}
\item $X$ admits collars 
\beastar
([-\epsilon,0] \times Y^+,\{0\} \times Y^+) & \hookrightarrow & (X,\del_+ X), \\
([0, \epsilon] \times Y^-, \{0\} \times Y^-) & \hookrightarrow & (X,\del_-X).
\eeastar
\item There are contact forms $\lambda^\pm$ on $Y^\pm$ respectively such that
$$
i_\pm^*\omega = d\lambda^\pm, \quad \xi^\pm = \ker \lambda^\pm
$$
for the embeddings $i_\pm: Y^\pm \to \del X$, and 
\beastar L \cap ( Y^+ \times [-\epsilon,0]) & = & \vec R^+ \times [-\epsilon,0],\\
L \cap (Y^- \times [0,\epsilon]) & = & \vec R^- \times [0,\epsilon].
\eeastar
\end{enumerate}
We say they are
\emph{exact Lagrangian cobordant} if the following additionally holds:
\begin{enumerate}
\item $\omega = d\lambda$ form some one-form $\lambda$ such that $i_\pm^*\lambda = \lambda^\pm$
respectively.
\item there exists a smooth function $f:L \to \R$ such that $\lambda|_L = df$ and $f$
is constant on $L \cap \del_-X$.
\end{enumerate}
\end{defn}
We would like to mention that we do not rule out the possibility that $\vec R^- = \emptyset$ when $\del L \neq \emptyset$.
The following remark is also worthwhile to be made.

\begin{rem} The conditions $i_\pm^*\lambda = \lambda^\pm$, $\lambda|_L = df$ 
and that $R^\pm$ are Legendrian imply that $f$ must be locally constant on $R^\pm$.
\begin{enumerate}
\item
Then the requirement $f$ begin constant on $L \cap \del_-X'$ is not a loss of any generality if $R^-$ is connected.
For a general Legendrian link $R^-$, we require the local constants be the same over different connected 
components of $R^-$.
\item The values of $f$ on $R^+$ are also locally constant but the local constants will no longer 
be the same componentwise, and we do not require the constants be the same on $R^+$ when we have already made 
the above requirement on $R^-$. 
\end{enumerate}
\end{rem}

\part{Construction of Legendrian CI Fukaya category}

We shall construct the filtered LCI $A_\infty$ category as
a strictly unital and weakly unobstructed filtered $A_{\infty}$ category ${\mathfrak C}$
defined over the group ring $\mathbb K[H_2(Y,\vec R)]$  with the ground ring 
$\mathbb K = \Z, \, \Z_2, \C$ or $\R$
depending on the circumstances. We largely utilize the framework of
\cite[Chapter 7]{oh:book-kias} by adapting it to the current contact/Legendrian case,
the former of  which in turn closely follows that of  \cite{fooo:anchored}, \cite{afooo}, 
\cite{fukaya:immersed} given in the symplectic/Lagrangian case.

\section{The cocycle problems and the abstract index}
\label{sec:cycle-problem}

The discussion on general compact (or tame) Legendrian submanifolds of 
arbitrary length enter in our construction of LCI Fukaya category.
Recall that unlike the general Lagrangian Floer theory, the action functional
$$
\gamma \mapsto \int \gamma^*\lambda
$$
associated to the Legendrian submanifolds is always single-valued, and
its values depend only on the path homotopy class in $\pi_1(Y, R)$. The latter is because
we have $\lambda$  restricted to any Legendrian submanifold vanishes.
This makes the various actions or the action filtrations entering in the study 
of the moduli space of pseudo-holomorphic strips or polygons not depend
on the choice of relevant holomorphic discs.

On the other hand the corresponding problems for the degree
(dimension) and for the orientation of the moduli space of
pseudo-holomorphic strips or polygons still depend on such discs. Some elaboration of this point of
compatibility is now in order.

The moduli space of strip-like contact instantons entering in the relevant counting
problem is an appropriate compactification of the solution space of \eqref{eq:contacton-bdy}.
To define these structures that are compatible with this compactification of the moduli spaces,
 we need to study:
\begin{enumerate}
\item {($\Z$-Grading):} a $\Z$-grading with respect to which $\partial$ has degree 1
\item {(Sign):} a sign on the generators which induces a
$\Z$-module (or at least $\Q$ vector space) structure on
$CI(R,R')$ with respect to which $\partial$ is a $\Z$-module
(resp. a $\Q$ vector space) homomorphism.
\end{enumerate}
In all of these structures, the relative version, i.e., the
`difference' between two generators $\gamma, \, \gamma' \in \Reeb(R,R')$ is canonically defined. 
For given pair $(R,R')$ of \emph{connected} Legendrian submanifolds we consider the set
$$
\pi_2(\gamma,\gamma') = \pi_2(\gamma,\gamma';R,R')
$$
of relevant relative homotopy classes, which are associated to the disc-maps with 2 punctures.

Then the main problem is to solve the following cocycle problem.

\begin{ques}[Cocycle problem]
Denote any of the above differences by $I_2(p,q;u)$ over the path $u:[0,1] \to \Omega(L,L')$
connecting $\gamma,\, \gamma' \in \Reeb(R,R')$.
Solve the cocycle problem of constructing a family of functions $I_1 =\{I_1(\gamma)\}_{\gamma \in \Reeb(R,R') }$
\emph{independent of} the choice of $u \in \CM(\gamma,\gamma';R,R')$ so that the identity
\be\label{eq:I(q)}
I_2(\gamma,\gamma':u) = I_1(\gamma) - I_1(\gamma')
\ee
holds for all choices of $\gamma,\, \gamma' \in \Reeb(R,R')$.
\end{ques}

Handling the aforementioned cocycle problem and amplifying the compatibility
requirement to the whole collection of arbitrary lengths
$$
I = \{I_k(\gamma_0, \ldots, \gamma_k;u)\}_{k \geq 1}, \quad u \in \CF(Y,R; \vec \gamma)
$$
for $R= (R_0, \ldots, R_k)$ and $\vec \gamma= (\gamma_0,\ldots, \gamma_k)$ 
with $\gamma_i \in \mathfrak{X}(R_i, R_{i+1})$, we employ the notion of \emph{abstract index} 
the explanation of which is now in order. The symplectic analog to such a notion was
utilized in \cite{fooo:anchored}.

We first recall the standard notion of nondegeneracy of Reeb chords.

\begin{defn}\label{defn:nondegeneracy-chords}
We say a Reeb chord $(\gamma, T)$ of $(R_0,R_1)$ is nondegenerate if
the linearization map $\Psi_\gamma = d\phi^T(p): \xi_p \to \xi_p$ satisfies
$$
\Psi_\gamma(T_{\gamma(0)} R_0) \pitchfork T_{\gamma(1)} R_1  \quad \text{\rm in }  \,  \xi_{\gamma(1)}
$$
or equivalently
$$
\Psi_\gamma(T_{\gamma(0)} R_0) \pitchfork T_{\gamma(1)} Z_{R_1} \quad \text{\rm in} \, T_{\gamma(1)}M.
$$
\end{defn}
Here $\phi^t_{R_\lambda}$ is the flow generated by the Reeb vector field $R_\lambda$.
We note that this nondegeneracy is equivalent the nondegeneracy of Reeb chords between $R_i$ and $R_j$.
(See \cite[Corollary 4.8]{oh:entanglement1}.)

We will also need a similar set associated to the disc-maps with 1 puncture. For this purpose,
we utilize the already chosen base path $\ell \in \Omega(R,R')$, 
called a \emph{bridge} in \cite{oh-yso:index}, 
and define the set of path homotopy classes of the maps
$
w:[0,1]^2 \to M
$
satisfying the boundary condition
\be\label{eq:condw}
w(0,t) = \ell_{01}(t), \quad w(1,t)\equiv \gamma(t), \quad w(s,0) \in R, \quad w(s,1) \in R'.
\ee
We denote the corresponding set of homotopy classes of the maps $w$ by
$$
\pi_2(\ell;\gamma).
$$
Then we have the obvious gluing map
\be\label{eq:pi2pell0}
\pi_2(\ell;\gamma) \times \pi_2(\gamma,\gamma') \to \pi_2(\ell;\gamma'); \, (\alpha,B) \mapsto \alpha \# B.
\ee

Now we would like to generalize this construction to the case of  Legendrian links
$R = (R_0,\cdots, R_k)$ of connected components $R_i$'s with length $k \geq 2$. 
The map \eqref{eq:pi2pell0} naturally extends to the polygonal map
\be\label{eq:pi2-polygonal}
\pi_2(\ell_{01};\gamma_1) \times \cdots \times \pi_2(\ell_{k-1k};\gamma_k) \times
\pi_2(\gamma_1,\cdots, \gamma_k) \to \pi_2(\ell_{0k};\gamma_0)
\ee
in an obvious way which we write by
\be\label{eq:gluing-rule}
(\alpha_1,\cdots, \alpha_k; B) \mapsto (\alpha_1 \# \cdots \# \alpha_k) \# B.
\ee
Let $R=(R_1,\cdots,R_k)$ be a Legendrian link. We assume that we have
$$
Z_{R_i} \pitchfork R_j
$$
for all $i, \, j= 1,\ldots, k$.

\begin{hypo}[General position]
\label{hypo:general-position}
We  assume the following general position conditions for the Legendrian link
$R= (R_0,\cdots, R_k)$ :
\begin{enumerate}
\item They are pairwise transversal in the above sense for each $(R_i, R_j)$ for all pairs
$(i,j)$ with the case $i = j$ included.
\item No three connected components carry Reeb chords that overlap on an open arc.
\end{enumerate}
\end{hypo}

Let $\mathfrak{R} = (R^0,\cdots,R^k)$ be a chain of Legendrian links and $\gamma_{i(i+1)}
\in \mathsf R^i \cap \mathsf R^{i+1}$. ($\gamma_{(k+1)k} = \gamma_{0k}$ and $\mathsf R^{k+1} 
= \mathsf R^0$ as convention.)
We write 
$$
\vec \gamma = \left(\gamma_{01},\cdots,\gamma_{(k-1)k}\right).
$$
For given boundary marked points $(z_0, \cdots z_k) \subset \del \Sigma$ with $\Sigma = D^2$,
we consider the punctured bordered Riemann surface $\dot \Sigma = \Sigma \setminus \{z_0, \cdots, z_k\}$
equipped with strip-like coordinates $\pm [0, \infty)$ near $z_i$, we
consider the set of homotopy class of maps
$w : \dot \Sigma  \to M$ satisfying
$$
w(\overline{z_i z_{i+1} }) \subset R_i, \quad
w((\infty,t)_{i(i+1)}) = \gamma_{i(i+1)}.
$$
We denote it by $\pi_2(R;\vec \gamma)$.

The following definition is a contact-Legendrian counterpart of 
the notion which was introduced  in relation to the construction of the filtered Fukaya category
\cite{fooo:anchored}. We also refine its definition by one the one hand
removing the unnecessary notion of \emph{admissibility} therefrom, and on the other
hand considering a \emph{monoids} as its codomain instead of just \emph{modules}.
\begin{defn}\label{defn:abstract-index}
Let $(\K,\circ)$ be a monoid. We say a collection of maps
$$
I = \{I_k : \pi_2(\CE;\vec \gamma) \to \K\}_{k=1}^\infty
$$
an {\it abstract index} over the collection of
graded bridged Legendrian links 
$$
\CE = (\CE^0, \cdots, \CE^k),
$$
if they satisfy the following gluing
rule: Under the gluing map \eqref{eq:pi2-polygonal}, we have
$$
\circ_{i=0}^{k-1} \left(I_1([w _{i(i+1)}])\right) \circ I_{k+1}(B) = I_1(([w _{01}]\# \cdots \# [w _{(k-1)k}]) \# B).
$$
\end{defn}

\section{Polygonal Maslov index}
\label{sec:polygonal}

For the following discussion, we consider the cases $k \geq 1$,
i.e, $\operatorname{length} \CE \geq 2$.
For each given such chains, we define 
$$
C^\infty(D^2,\CE;\vec \gamma), \quad \vec \gamma =\{\gamma_{01},\gamma_{12},\cdots, \gamma_{0k}\}
$$
to be the set of all $w: \dot \Sigma \to M$ such that
\be\label{eq:wzjpj}
w(\overline{z_{(j-1)j}z_{j(j+1)}}) \subset R_j, \quad w (\infty_j,t) = \gamma_{(j-1)j} \in 
\mathfrak{X}(R_{j-1}, R_j)
\ee
and that it is continuous on $D^2$ and smooth on $\dot D^2$. We will
define a topological index, which is associated to each homotopy
class $B \in \pi_2(\mathsf R;\vec \gamma)$.
For this purpose, we use the notion of \emph{graded bridged Legendrian links}.

\begin{defn}[Polygonal Maslov index]\label{defn:polygonal}
Let $\CE=(R_0, \cdots, R_k)$ be a Legendrian link in general position
(Hypothesis \ref{hypo:general-position}).
We define the topological index, denoted by $\mu(\CE, \vec \gamma;B)$,
to be the Maslov index of
the loop $\widetilde\alpha_w$, i.e.,
$$
\mu(\CE, \vec \gamma;B) = \mu(\widetilde\alpha_w).
$$
\end{defn}

Now consider a graded bridged Legendrian link
$\CE = (R_0, \cdots, R_k)$.
By definition, it assigns a grading $ \lambda_{ij} $ along $\ell_{ij}$.

\begin{lem}\label{thm:poly} Let $\CE$ be a graded bridged Legendrian link. Let
$[\ell_{i(i+1)},w_{i(i+1)}]$ for $k = 0, \ldots, k-1$ and  $B \in \pi_2(\CE,\vec \gamma)$ be
given. Consider the class $[\gamma_{0k},w_{0k}]$ be the one determined by
\eqref{eq:gluing-rule}. Then we have
\begin{equation}\label{musum}
\mu([\gamma_{0k},w_{0k}]) = \mu(\CE,\vec \gamma;B) + \sum_{i=0}^{k-1}
\mu([\gamma_{i(i+1)},w_{i(i+1)}];\lambda_{i(i+1)})
\end{equation}
\end{lem}
\begin{proof} Let $\gamma_{k0} = \gamma_{0k} \in L_0 \cap L_k$ and consider the $s$ time-reversal
path $w_{k0}$ defined by $w_{k0}^+(s,t): = w_{0k}(1-s,t)$.
Since $\mu([\gamma_{k0},w_{k0}^+];\lambda_{0k})$ is defined as the Maslov index of
the time-reversal loop of $\alpha_{([\gamma_{0k},w_{0k}];\lambda_{0k})}$,
the equality (\ref{musum}) follows from
$$
\sum_{i=0}^{k-1} \alpha_{([\gamma_{i(i+1)},w_{i(i+1)]};\lambda_{i(i+1)})}
+ \alpha_{([\gamma_{k0},w_{k0}^+];\lambda_{k0}^+)} +  \widetilde{\alpha}_w \sim 0,
$$
where $\widetilde\alpha_w$
is as in Definition \ref{defn:grading-chord}
with $B = [w_{01}]\# \cdots \# [w_{(k-1)k}] \# [w_{k0}]$ and
$\sim$ means homologous.
\end{proof}

When the length of $\CE$ is $k+1$, we define
$
\mu_k(B) = \mu(\CE,\vec \gamma;B)
$
where $B \in \pi_2(\CE;\vec \gamma)$.

\begin{prop} Let $R, \, R'$ be a two connected Legendrian submanifolds and
$(\ell,\alpha)$ be a graded bridge between them. Then we
define $\mu_1:\pi_2(\ell,\gamma) \to \Z$ with $\Z$ regarded as an additive
monoid by setting
$$
\mu_1(\alpha) : = \mu([\gamma,w])
$$
for a representative $[\gamma,w]$ of the class $\alpha \in \pi_2(\ell; \gamma)$.

Then the sequence of maps $\mu = \{\mu_k\}_{k=1}^\infty$ with
$$
\mu_k: \pi_2(\CE;\vec \gamma) \to \Z, \quad k = \text{\rm leng}(\vec \gamma) \geq 1
$$
defines an abstract index with $(\Z,+)$ regarded as an additive monoid.
\end{prop}

\begin{rem} While we consider the bridges between two connected components of 
a given Legendrian link, we could also consider the more general cases of 
two different Legendrian links and replacing $R_i$ by
the support $\mathsf R^j$ for a tuple of Legendrian links as in 
\cite{oh-yso:index}. For the simplicity of exposition here, we restrict to the current case
leaving a full discussion elsewhere. Such a generalization will be the playground
for the study of more complex \emph{entanglement} structure. The relevant formalism 
would be as follows.

We consider a chain of Legendrian links $\mathfrak R = (\mathsf R^0, \cdots, \mathsf R^k)$ and
a chain of Reeb chords $(\gamma_{01},\gamma_{12},\cdots, \gamma_{0k})$ with 
$$
\gamma_{(i-1)i} \in \mathfrak{X}(\mathsf R^{i-1}, \mathsf R^i)
$$
for $i=0, \cdots, k$ (counted modulo $k+1$). We consider $\Sigma \cong D^2$
with marked points $\{z_{01},z_{12},\cdots, z_{0k}\}$ ordered counter-clock-wise
and denote $\dot \Sigma$. We denote by $\CE$ the chain of graded bridged Legendrian links
associated to $\mathfrak R$.

For each given such chains, we define 
$$
C^\infty(D^2,\CE;\vec \gamma), \quad \vec \gamma =\{\gamma_{01},\gamma_{12},\cdots, \gamma_{0k}\}
$$
to be the set of all $w: \dot \Sigma \to M$ such that
\be\label{eq:wzjpj}
w(\overline{z_{(j-1)j}z_{j(j+1)}}) \subset \mathsf R^j, \quad w(\infty_j,t)= \gamma_{(j-1)j} \in 
\mathfrak{X}(\mathsf R^{j-1}, \mathsf R^j)
\ee
and that it is continuous on $D^2$ and smooth on $\dot D^2$.

\end{rem}

\section{Coherent orientations of the moduli spaces}
\label{sec:orientation}

For the purpose of  defining various operators entering in the construction of 
Fukaya-type $A_\infty$ category, we need to provide a
compatible system of orientations on the moduli space of strip-like contact instantons and
other moduli spaces of polygonal punctured bordered contact instantons. 

We will show that coherent orientations can be packaged into an example of 
\emph{abstract index} (Definition \ref{defn:abstract-index}) to the context of \emph{Legendrian
chains} instead of Lagrangian chains. (See \cite{fooo:anchored}
for the relevant discussion for the Lagrangian intersection Floer theory.)
For this purpose, we will  describe the codomain monoid of this abstract index.

We start with some discussion of relatively spin Legendrian submanifolds.

\subsection{Relatively spin Legendrian submanifolds}

Recall that we have a canonical exact sequence 
$$
0  \to \xi \to TM \to TM/\xi \to 0
$$
of vector bundles, and when $(M,\xi)$ is coorientable, $TM/\xi$ is trivial.
With a choice of coorientation of $\xi$, $TM/\xi$ is trivialized and so 
trivialization of $\xi$ naturally provides one for $TM$ extending the trivialization of $TM/\xi$.
When a (postive) contact form $\lambda$ of cooriented $\xi$ is given,
 the above exact sequence splits
$$
TM = \xi \oplus \R\langle R_\lambda \rangle
$$
where $R \langle R_\lambda \rangle \cong TM/\xi$ as a trivialized line bundle. 

We now recall  the notion of relative spin submanifold applied to the Legendrian submanifold.

\begin{defn} Suppose $(M,\xi)$ is a contact manifold equipped with coorientation.
We say a Legendrian submanifold $R \subset (M,\xi)$ is \emph{relatively spin} for $(M,\xi)$
if it is orientable and there exists a class $\text{\rm st} \in H^2(M,\Z_2)$ such that
$\text{\rm st}|_R = w_2(TR)$ for the Stiefel-Whitney class $w_2(TR)$ of $TR$.
\par
A chain $(R_0,R_1,\cdots,R_k)$ or a pair $(R_0,R_1)$ of Legendrian submanifolds is
said to be \emph{relatively spin} if there exists a class
$\text{\rm st} \in H^2(M,\Z_2)$ satisfying $\text{\rm st}|_{R_i} = w_2(TR_i)$ for each
$i = 0, 1, \cdots, k$.
\end{defn}

We fix such a class $\text{\rm st} \in H^2(M,\Z_2)$ and a  relative triangulation of $(M, \mathsf R)$.
Denote by $M^{(k)}$ its $k$-skeleton. There exists a
real subbundle $V(\text{\rm st})$ on $M^{(3)}$ with
$w_1(V(\text{\rm st})) = 0, \, w_2(V(\text{\rm st})) = \text{\rm st}$. Now suppose that $R$ is
relatively spin and $R^{(2)}$ be the 2-skeleton of $R$.
Then $V \oplus TR$ is trivial on the 2-skeleton of $R$. 

\begin{defn} We define a $(M,\text{\rm st})$-relative
spin structure of $R$ to be a choice of $V$ and a spin structure of the
restriction of the vector bundle $V \oplus TR$ to $R^{(2)}$.
\par
The relative spin structure of a chain of Legendriian submanifolds
$(R_0,\cdots,R_k)$ is defined in the same way by
using the same $V$ for all $R_i$.
\end{defn}

\subsection{Orientatations $o_{\gamma}$, $o_{[\gamma;w]}$ and $o_{(\gamma,\gamma';B)}$}

Let $\gamma, \gamma' \in \mathfrak{X}(R,R')$ and $B \in \pi_2(\gamma,\gamma')$. 
We consider finite energy, (i.e.,
 $E(u) = E^\pi(u) + E^\perp(u)  < \infty$) solutions $u
: \R\times [0,1] \to M $ of the equation 
\be\label{eq:contacton-bdy}
\begin{cases}
\delbar^\pi u = 0, \, d(u^*\lambda \circ j) = 0 \\
u(\tau,0) \in R_0, \quad u(\tau,1) \in R_1,  \\
u(-\infty,\cdot) \equiv \gamma, \quad u(\infty,\cdot) \equiv \gamma'.
\end{cases}
\ee
\par
If $(R_0, R_1)$ is a relatively spin pair, then each $\CM(\gamma,\gamma';B)$ is
orientable. Furthermore a choice of relative spin structures gives
rise to a compatible system of orientations for $\CM(\gamma,\gamma';B)$ for all
pair $\gamma, \, \gamma' \in \mathfrak{X}(R_0,R_1)$ and $B \in \pi_2(\gamma,\gamma')$.

We remark that relative spin structure determines a trivialization of 
$$
V_{\lambda_{01}(0)} \oplus T_{\lambda_{01}(0)} R_0 =
V_{\lambda_{01}(0)} \oplus \lambda_{01}(0)
$$
and 
$$
V_{\lambda_{01}(1)} \oplus T_p R_1 = V_{\lambda_{01}(1)} \oplus \lambda_{01}(1).
$$
We take and fix a way to extend this trivialization to the family
$\ell_{01}^*V \oplus \lambda_{01}$ on $[0,1]$.

We put
\be
Z_+ = \{ (\tau,t) \in \R^2 \mid \tau \le 0, \,\, 0\le t\le 1\} 
\ee
and the projection $\pi: Z_+ \to [0,1]$. We 
consider $W^{1,p}$ sections $\eta \in (\gamma\circ \pi)^*TM$ satisfying
the linearized contact instanton equation
and the relevant boundary condition. It defines a linear Fredholm operator
$$
D\Upsilon(u;(\gamma;\lambda)) : W^{1,p}((\gamma \circ \pi)^*TM;\lambda) 
\to L^p((\gamma \circ \pi)^*TM \otimes \Lambda^{0,1}).
$$
It defines virtual vector space
$$
\Index D\Upsilon(u;(\gamma;\lambda))
$$
and the determinant line $\det \Upsilon_{(\CE;\vec \gamma;B)}$. It defines an 
orientation which denote by $o_\gamma$.

Moving $u$ we obtain a family of Fredholm operator $D\Upsilon(u;(\gamma;\lambda))$ parameterized by a suitable completion of the off-shell space
$\CF(\gamma,\gamma';B)$ for $B \in \pi_2(\gamma, \gamma' ; R,R')$. Therefore we have
a well-defined determinant line bundle
 \be\label{eq:detB1} \det
\Upsilon_{(R,R';\gamma, \gamma';B)} \to \CF(\gamma,\gamma';B). 
\ee

The following theorem can be proved in the same way as in \cite[Chapter 8]{fooo:book2}.
\begin{thm}\label{thm:1-puncture-ori}
Let $(R_0,R_1)$ be a relatively spin pair of oriented Legendrian
submanifolds. Then the following hold:
\begin{enumerate}
\item For each fixed $\alpha$ the bundle
\eqref{eq:detB1} is trivial.
\item
If we fix a choice of system of
orientations $o_{\gamma}$ on $\operatorname{Index}\,D\Upsilon(u;(\gamma;\lambda))$
for each $\gamma$, then it
determines orientations on \eqref{eq:detB1}, which we denote by $o_{[\gamma,w]}$.
\item 
Moreover $o_{\gamma}$, $o_{[\gamma,w]}$ determine an orientation of
$\CM(p\gamma,\gamma';B)$ denoted by $o(\gamma,\gamma';B)$ by the gluing rule
\be\label{eq:ori-gluing}
o_{[\gamma',w\#B]} = o_{[\gamma,w]} \# o(\gamma,\gamma';B)
\ee
for all $\gamma, \, \gamma' \in \frak{X}(R,R')$ and $B \in \pi_2(\gamma,\gamma')$ so that
they satisfy the gluing formulae
$$
\del o(\gamma,\gamma'';B) = o(\gamma,\gamma';B_1) \# o(\gamma',\gamma''';B_2)
$$
whenever the virtual dimension of $\CM(\gamma,\gamma'';B)$ is $1$.  Here
$\del o(\gamma,\gamma'';B)$ is the induced boundary orientation of
the boundary $\del \CM(\gamma,\gamma'';B)$ and
$B = B_1 \# B_2$ and $\CM(\gamma,\gamma';B_1) \# \CM(\gamma',\gamma'';B_2)$ appears
as a component of the boundary $\del \CM(\gamma,\gamma'';B)$.
\end{enumerate}
\end{thm}

\subsection{Construction of coherent orientations}

We generalize the discussion on the moduli spaces of
$\overline \CM(\gamma,\gamma')$ of contact instanton trajectories to the moduli space of polygons,
whose explanation is in order.

Consider a disc $D^2$ with $k+1$ marked points $z_{01}, z_{12}, \cdots, z_{0k}
\subset \del D^2$ respecting the counter clockwise cyclic order of
$\del D^2$. We take a neighborhood $U_i$ of $z_{i(i+1)}$ and a conformal
diffeomorphism
$$
 \varphi_i: U_i \setminus \{z_{i(i+1)}\} \subset D^2 \cong
(-\infty,0] \times [0,1]
$$
 of each $z_{i(i+1)}$.
For any smooth map
$$
w: D^2 \to M; \quad w(z_{i(i+1)}) = \gamma_{i(i+1)}, \, w(\overline{z_{(i-1)i}z_{i(i+1)}}) \subset R_i
$$
we deform $w$ so that it becomes constant on
$\varphi_i^{-1}((-\infty,-1] \times [0,1]) \subset U_i$, i.e., $ w(z)
\equiv \gamma_{(i-1)i} $ for all $z \in \varphi^{-1}((-\infty,-1] \times
[0,1])$. So assume this holds for $w$ from now on. We now consider
the linearlized $\Upsilon$ equation
under  the boundary condition
\eqref{eq:wzjpj} becomes
\be\label{eq:linearCR}
\xi(s,0) \in TR_{i-1}, \quad \xi(s,1) \in TR_{i}.
\ee
\par
The boundary value problem \eqref{eq:linearCR} induces a Fredholm operator,
which we denote by
\be\label{delbawl}
D\Upsilon_{w;\CE} :
W^{1,p}(D^2;w^*TM;\CE) \to L^p(D^2;w^*TM \otimes \Lambda^{0,1}).
\ee
Moving $w$ we obtain a family of Fredholm operator $D\Upsilon_{w;\CE}$ parameterized by a suitable completion of $\CF(\vec
p;\CE;B)$ for $B \in \pi_2(\vec \gamma ; \CE)$. Therefore we have
a well-defined determinant line bundle
 \be\label{eq:detB} \det
\Upsilon_{(\CE;\vec \gamma;B)} \to \CF(\CE;\vec \gamma;B). 
\ee 
The following theorem an analog to  \cite[Theorem 6.7]{fooo:anchored}
which can be proved in a similar way.
\begin{thm} \label{thm:fooo-ori2}
Suppose $\CE = (R_0,\cdots, R_k)$ is a relatively spin graded bridged Legendrian link.
(See Definition \ref{defn:graded-bridges}.)
Then the following hold:
\begin{enumerate}
\item
Each determinant line bundle $\det
D\Upsilon_{(\vec \gamma;\CE;B)}$ is trivial.
\item
If we fix orientations $o_{\gamma_{ij}}$
on $\operatorname{Index}\,D\Upsilon_{(\vec \gamma;\CE;B)}$ as in Theorem \ref{thm:1-puncture-ori}
for all $\gamma_{ij} \in \mathfrak{X}(R_i,R_j)$, with $R_i$
transversal to $R_j$, then we have a
system of orientations, denoted by $o_{k+1}(\vec \gamma;\CE;B)$, on the
bundles $(\ref{eq:detB})$.
\end{enumerate}
\end{thm}
\begin{proof}
Let $w_{i(i+1)} :[0,1]^2 \to M$  be the map given as in \eqref{eq:condw},
and we consider the operator
$D\Upsilon_{([\gamma_{(i+1)i},w_{i(i+1)}];\lambda_{i(i+1)})}$. We glue it
with $ D\Upsilon_{(\CE;\vec \gamma;B)}$ at $U_{i+1}$ (The boundary condition \eqref{eq:linearCR}
enables us to glue the boundary condition.) After gluing all
of $D\Upsilon_{([\gamma_{i(i+1)},w_{i(i+1)}];\lambda_{i(i+1)})}$ we have
an index bundle of a Fredholm operator
\be\label{summedupindex}
D\Upsilon_{(\CE;\vec \gamma;B)} \#
\left(\sum_{i=0}^{k-1} D\Upsilon_{([\gamma_{i(i+1)},w_{i(i+1)}];\lambda_{i(i+1)})}\right)
\# D\Upsilon_{([\gamma_{k0},w^+_{k0};\lambda_{0k}])}.
\ee
By \cite[Lemma 3.7.69]{fooo:book2}, the index bundle of (\ref{summedupindex})
has canonical orientation.
On the other hand, the index virtual vector spaces
$$
\text{\rm Index}\, D\Upsilon_{([\gamma_{i(i+1)},w_{i(i+1)}];\lambda_{i(i+1)})}
$$
are oriented by Theorem \ref{thm:1-puncture-ori}.
Theorem \ref{thm:fooo-ori2} follows.
\end{proof}

We can prove that the orientation of $\text{\rm Index } D\Upsilon_{(\CE;\vec \gamma;B)}$
depends on the choice of $o_{\gamma_{i(i+1)}}$ (and so on $\lambda_p$)
with $i=0,\cdots,k-1$ but is independent of the choice of $w_{i(i+1)}$ etc.
Therefore the orientation in Theorem \ref{thm:fooo-ori2} is independent of the choice of
base paths.
(See  \cite[Remark 8.1.15 (3)]{fooo:book2} for elaboration of this point.)

Now we consider the system of maps
$$
o_k: \pi_k(\mathfrak \CE) \to \Z_2, \quad k \geq 1
$$
with $o_1(\alpha) = \det D\Upsilon_{[\ell,\gamma]}$ as an element of
the set of orientations.
We mentioned that the set of orientations is \emph{non-canonically} isomorphic to $\Z_2$.

\begin{cor} \label{cor:fooo-ori2}
Suppose $\CE = (R_0,\cdots, R_k)$ is a relatively spin Legendrian link.
Then the system $\mathfrak{Or}: = \{o_k\}_{k=1}$ defines an abstract index.
\end{cor}

\begin{rem}\label{fooo8-5}
In order to give an orientation of $\CM(\CE;\vec \gamma;B)$, we have to take the moduli
parameters of marked points and the action of the automorphism group into account.
See \cite[Section 8.3]{fooo:book2} for detailed explanation.
\end{rem}

\section{Construction of filtered Legendrian CI Fukaya category}
\label{sec:filtered-fukaya}

In this section we construct a curved $A_\infty$ category of a tame contact manifold
such that its objects consist of Legendrian submanifolds with suitable decorations and
its structure maps are given by the structure maps ${\mathfrak m}_k$ for $k \geq 0$
that satisfy the $A_\infty$ relations.
We start with  recalling the definition of the chain module
$$
C(R,R';\ell;\K[q,q^{-1}) \cong \bigoplus_{\gamma \in \mathfrak{X}(R,R')} 
\K[q,q^{-1}] \langle \gamma \rangle
$$
from \eqref{eq:CRR'ell}.

 The following is an immediate consequence of 
the definition and the compactness theorem
of the moduli space of contact instantons \cite{oh:entanglement1}.

\begin{lem} Let $J$ be a compatible CR almost complex structure of $(M,\lambda)$,
and $(R,R')$ be a pair of compact Legendrian submanifolds. 
\begin{enumerate}
\item  $\Gamma_{RR'} \subset  2\Z$ is  a \emph{$(\Gamma,\Gamma')$-set} where 
$\Gamma = \Gamma(R,J), \, \Gamma'= \Gamma_{R',J}$.
\item The structure maps  $\mathfrak n = \{{\mathfrak n}_{k_0,k_1}\}$ are $\Gamma_{RR'}$-gapped.
\end{enumerate}
\end{lem}

The following structure theorem is the main consequence which is
 an algebraic translation of the boundary 
structure of the compactified moduli spaces
$$
\overline \CM_{k_0,k_1} (R,R';[w,\gamma],[w',\gamma']), \quad k_0, \, k_1 \geq 0
$$
especially for those of dimension 1.

\begin{thm}[$A_\infty$ bimodule structure]\label{thm:bimodule}
Let $(R,R')$ be a nondegenerate relatively
spin pair of Legencrianl submanifolds.
We have a left $(C(R;\K[q,q^{-1}]),{\mathfrak m})$
and right $( C(R';\K[q,q^{-1}]),\mathfrak
m^{\prime})$ filtered $A_{\infty}$ bimodule structure on
$C(R,R';\K[q,q^{-1}])$.
\end{thm}

Now we let $k \geq 2$ and consider the case of Legendrian link
$$
\CE = (R_0, R_1, \cdots, R_k)
$$
which is a
chain of (compact) Legendrian submanifolds in $(M,\omega)$
that intersect pairwise transversely without triple intersections.

\par
Let $\vec z = (z_{01},z_{12},\cdots,z_{(k-1)k})$ be a set of distinct points on $\partial D^2
= \{ z\in \C \mid \vert z\vert = 1\}$. We assume that
they respect the counter-clockwise cyclic order of $\partial D^2$.
The group $PSL(2;\R)\cong \operatorname{Aut}(D^2)$ acts on the set
in an obvious way. We denote by $\mathcal M_{k+1}$ be
the set of $PSL(2;\R)$-orbits of $(D^2,\vec z)$.
Recall that there is no automorphism on the domain $(D^2, \vec z)$, i.e.,
$PSL(2;\R)$ acts freely on the set of such $(D^2, \vec z)$'s, when $k \geq 2$.

Let
$\gamma_{(j-1)j} \in \mathfrak{X}(R^{j-1},R^j)$
($j = 0,\cdots k$), be a set of intersection points.
\par
We consider the pair $(w;\vec z)$ where $w: D^2 \to M$ is a
pseudoholomorphic map that satisfies the boundary condition
\bea\label{54.15}
&{}& w(\overline{z_{(j-1)j}z_{j(j+1)}}) \subset R_j, \nonumber \\
&{}& w((\infty,t)_j) = \gamma_{j(j+1)}(t) \in \mathfrak{X}(R_j,R_{j+1}).
\eea
We denote by $\widetilde{\CM}(\CE, \vec \gamma)$
the set of such
$((D^2,\vec z),w)$.
\par
We identify two elements $((D^2,\vec z),w)$, $((D^2,\vec z'),w')$
if there exists $\psi \in PSL(2;\R)$ such that
$w \circ \psi = w'$ and $\psi(z'_{j(j-1)}) = z_{j(j-1)}$.
Let ${\CM}(\CE, \vec \gamma)$ be the set of equivalence classes.
We compactify it by including the configurations with disc or sphere bubbles
attached, and denote it by $\overline{\CM}(\CE, \vec \gamma)$.
Its element is denoted by $((\Sigma,\vec z),w)$ where
$\Sigma$ is a genus zero bordered Riemann surface with one boundary
components, $\vec z$ are boundary marked points, and
$w : (\Sigma,\partial\Sigma) \to (M,L)$ is a bordered stable map.
\par
We can decompose $\overline \CM(\CE, \vec \gamma)$ according to the homotopy
class $B \in \pi_2(\CE,\vec \gamma)$ of continuous maps satisfying
\eqref{54.15} into the union
$$
\overline \CM(\CE, \vec \gamma) = \bigcup_{B \in \pi_2(R;\vec \gamma)}
\overline \CM(\CE, \vec \gamma;B).
$$

\begin{thm}[Section 11.3, \cite{oh-yso:index}]\label{thm:dimension} 

Let $\CE = (R_0,\cdots,R_k)$ be a Legendrian link in a contact manifold $(Y,\xi)$ of
dimenison $2n+1$, and $B \in \pi_2(R;\vec \gamma)$.
Then $\CM(\CE, \vec \gamma;B)$ has its (virtual) dimension satisfies
\be\label{dimensionformula}
\dim \CM(\CE, \vec \gamma;B) = \mu(\CE,\vec \gamma;B) + n + k-2,
\ee
where $\mu(\CE,\vec \gamma;B)$ is the polygonal Maslov index of $B$.
\end{thm}

We next take graded base paths $(\ell_i,\lambda_i)$ to each $R_i$ for $i = 1, \cdots, k-1$
and assign the base path $\ell_{0k}$ to the pair $(R_0,R_k)$.

Using the case $\dim \CM(\CE, \vec \gamma;B) = 0$ and recalling the notation
\eqref{eq:CIRR'ell} of $CI(R,R';\ell)$ in general, we define the $k$-linear operator
\beastar
\mathfrak m_{k,B}:CI(R_0,R_{1};\ell_{01}) & \otimes &\cdots\otimes
CI((R_{k-1}R_{k};\ell_{(k-1)k}) \\
& \longrightarrow & CI(R_0,R_k; \ell_{0k})
\eeastar
by
\be\label{catAinifwob}
\aligned
\mathfrak m_{k,B}([\gamma_{01},w_{01}]),& [\gamma_{12},w_{12}],\cdots, [\gamma_{(k-1)k},w_{(k-1)k}]) \\
&= \sum \#(\CM_{k+1}(\CE;\vec \gamma;B)) \, [\gamma_{0k},w_{0k}])
\endaligned
\ee
when  $\ell_{0k} =  (\sum_{i=0}^{k-1} \ell_{i(i+1)}) \# B$.
Here the sum is over the basis $[\gamma_{0k},w_{0k}]$ of
$CI(R_k,R_0)$, where  $B \in \pi_2(\CE;{\bf p})$ and $
\vec \gamma = (\gamma_{01},\gamma_{12}, \ldots, \gamma_{(k-1)k},\cdots,\gamma_{k0})$ with $\gamma_{k0} = \gamma_{0k}$.
The formula (\ref{misdeg1}) implies that $\mathfrak m_{k,B}$ above is a $\K$-linear map
over the given coefficient ring (e.g., $\K = \R, \, \C$).

By the procedure similar to the case of `boundary map' $\delta = \mathfrak m_1$
in Theorem {thm:bimodules}, we can rewrite the equation into an $R$-linear\
map $\mathfrak m_{k,B}$ given by
$$
\mathfrak m_{k,B}(\langle \gamma_{01} \rangle, \ldots, \langle \gamma_{(k-1)k} \rangle)
=  \sum_{\gamma_{0k} \in R_0 \cap R_k}  \#(\CM_{k+1}(\CE;\vec \gamma;B)) \, \langle \gamma_{0k} \rangle
$$
linearly extended over $\Lambda_{0,nov}$, where we simplify the notation by writing
\be\label{eq:simplified-notation}
\langle \gamma_{(i-1)i} \rangle: = [\gamma_{(i-1)i},w_{(i-1)i}].
\ee
So far, we have constructed a curved filtered $A_\infty$ category
out of a collection of Legendrian submanifolds equipped with
a system of \emph{graded base paths}.

\section{Strictification}
\label{sec:strictification}

Our ultimate goal is to construct an $A_\infty$ category
as a contact invariant of  the underlying contact manifold
$(M,\omega)$. For this purpose, the \emph{strictification} is needed
so that the homology group of $(C(c_0,c_1),\mathfrak{m}_1)$
becomes invariant under the Hamiltonian isotopy. This is because
in general the operator $\mathfrak m_k$ defined in Section \ref{sec:filtered-fukaya}
does {\it not} satisfy the classical $A_{\infty}$ relation or equivalently $\frak m_0 \neq 0$
by the presence of obstructions as in the case of defining the boundary operator $\frak m_1$ satisfying
$\frak m_1^2 = 0$. We need to use bounding cochains of $R_i$ to deform $\mathfrak m_k$ in the same
way as the symplectic case in \cite{fooo:book1},  \cite{fukaya:immersed} whose explanation is now in order.
(Also see \cite[Subsection 7.1.3]{oh:book-kias} for the summary of the procedure of strictification.)

Let $m_0,\cdots,m_k \in \Z_{\ge 0}$ and $\CM_{m_0,\cdots,m_k}(\CE,
\vec \gamma;B)$ be the moduli space obtained from the set of $((D^2,\vec
z),(\vec z^{(0)},\cdots,\vec z^{(k)}),w)$ by taking the quotient by
$PSL(2,\R)$ action and then by taking the stable map
compactification as before. Here we put
$$ 
z^{(i)} = (z^{(i)}_1,\cdots,z^{(i)}_{k_i}), \quad z^{(i)}_{j} \in
\overline{z_{(i-1)i}z_{i(i+1)}}
$$ 
such that
$z_{(i-1)i},z^{(i)}_1,\cdots,z^{(i)}_{k_i}, z_{i(i+1)}$ respects the counter-clockwise
cyclic ordering. The assignment
$$
\left((D^2,\vec z),(\vec z^{(0)},\cdots,\vec z^{(k)}),w\right)
\mapsto \left(w(z^{(0)}_1),\cdots,w(z^{(k)}_{m_k})\right)
$$
induces an evaluation map:
$$
\textbf{ev} =(\ev^{(0)},\cdots,\ev^{(k)}): \CM_{m_0,\cdots,m_k}(\CE, \vec \gamma;B)
\to \prod_{i=0}^k R_i^{m_i}.
$$
Taking the evaluation map
$$
(\CM_{m_0,\cdots,m_k}(\CE, \vec \gamma;B), {\bf \ev})
$$
as a correspondence, we define the map
\beastar
&{}& \mathfrak m_{k, B;m_0,\cdots,m_k} : B_{m_0}(CI(R_0) \otimes
CI(R,R')
\otimes B_{m_1}(CI(R_1)) \cdots\\
& \otimes&  CI(R_{k-1},R_k,\gamma_{0}) \otimes B_{m_k}(CI(R_k))
\to CI(R_0,R_k)
\eeastar
again by the pull-push as before
\beastar
&{}& \mathfrak m_{k;m_0,\cdots,m_k}
\left(\vec b^{(0)},[\gamma_{01},w_{01}],\cdots,[\gamma_{0k},w_{0k}],
\vec b^{(k)}\right)
\\
& = & \sum \#(\CM_{k+1}(\CE;\vec \gamma;\vec{\vec b};B)) \, [\gamma_{0k},w_{0k}]
\eeastar
where ${\bf b}^{(i)} \in B_{m_i}(CI(R_i))$ with our choice of the chain model
 $CI(R_i) = \Omega^*(R_i)$  for the Morse-Bott  intersection $R_i \cap R_i$.
 We then define
 $$
 \mathfrak m_{k;m_0,\cdots,m_k}
 = \sum_{B} \mathfrak m_{k,B; m_0,\cdots,m_k}.
 $$
 by the same procedure as in  \cite{fooo:anchored}.

Finally summing over this collection of maps, we
define the \emph{boundary deformation} of $\mathfrak m = \{\mathfrak m_k\}_{k \geq 0}$
as follows.
For each given $b_i \in CI(R_i)[1]^0$ ($b_i \equiv 0 \mod \Lambda_+$),
$\vec b =(b_0,\cdots,b_k)$, and $x_i \in CI(R_{i-1},R_i; \ell_{(i-1)i})$, we put
\be\label{mkcorrected}
\mathfrak m_k^{\vec b}(x_1,\cdots,x_k) = \sum_{m_0,\cdots,m_k} \mathfrak
m_{k;m_0,\cdots,m_k} (b_0^{m_0},x_1, b_{2}^{m_{2}},\cdots,x_k,b_k^{m_k}).
\ee
Then we obtain the following in the same way as the symplectic analog
was proved in \cite{fooo:anchored}.
\begin{thm}
If every $b_i$ is a bounding cochain, i.e.,  satisfies Maurer-Cartan equation
 $$
 \sum_{i=0}^\infty {\mathfrak m}_k(b,\ldots, b) = 0,
 $$
 then $\mathfrak m_k^{\vec b}$ in $(\ref{mkcorrected})$ satisfies the
$A_{\infty}$ relation: $\forall$ $k\geq 1$, we have
\be\label{Ainftyrel}
\sum_{k_2 = 0}^{k}\sum_{i=0}^{k+1-k_2}
(-1)^* \mathfrak m_{k+1 - k_2}^{\vec b}\left(x_1,\cdots,x_i, \mathfrak m_{k_2}^{\vec b}(x_{i+1},\cdots,x_{i+k_2}),\cdots,x_k\right) = 0
\ee
for every $k \geq 0$.
(We write $\mathfrak m_k$ in place of $\mathfrak m^{\vec b}_k$ in $(\ref{Ainftyrel})$.)
The sign $*$ is given by
$$
* = i -1 + |x_1| +\cdots + |x_{i-1}| \equiv  |x_1|' + \cdots |x_{i-1}|'
$$
\end{thm}

We summarize the above discussion as follows. To visualize the Floer theory perspective
in highlight, we write
$$
CI(R,R'): = C(R,R';\K[q,q^{-1}])
$$
for the $\Lambda$-module
$$
C(R,R';\K[q,q^{-1}])\otimes \K[q,q^{-1}]
$$
where $C(R,R';\K[q,q^{-1}])$ is the $\K[q,q^{-1}]$-module
as defined in \cite{fooo:anchored}.   In particular the module $CI(R,R')$
is equipped with the \emph{energy filtration} defined in Definition \ref{defn:action-filtration}.
The operations $\mathfrak m_k^{\vec b}$ are compatible with the filtration in the following sense
which completes construction of a filtered $A_\infty$ category, i.e., filtered Fukaya category
$\mathfrak{Fuk}(M,\omega)$.

\begin{prop}\label{filprod}
If $x_i \in F^{\lambda_i}CI(R,R_{i+1})$,
then
$$
\mathfrak m_k^{\vec b}(x_1,\cdots,x_k)
\in F^{\lambda}CI(R_0,R_k)
$$
where
$
\lambda = \sum_{i=1}^{k} \lambda_i.
$
\end{prop}

In conclusion, we obatin the following results.

\begin{thm}\label{anchoredAinfty}
We can associate a filtered $A_{\infty}$ category to a
contact manifold $(M,\omega)$ such that:
\begin{enumerate}
\item {\bf Objects:}  $(R, b, \sigma)$ where
$L$ is an oriented Legendrian submanifold equipped with a bounding cochain $[b]
\in {\mathfrak M}(CI(R))$
and  a stable conjugacy class $\sigma$ of relative spin structure $\mathsf{st}$ of $R$.
\par
\item {\bf Morphism spaces:} $CI(R,R')$ (or its homology $HI(R,R')$).
\par
\item {\bf Structure maps:} $\{\mathfrak m^{\vec b}_k\}$ are the operators defined in
$(\ref{mkcorrected})$.
\end{enumerate}
\end{thm}

\bibliographystyle{amsalpha}

\bibliography{biblio2}

\end{document}